\documentclass{amsart} 
\usepackage{amssymb,amsmath}

\newtheorem{theorem}{Theorem}[section]
\newtheorem{lemma}[theorem]{Lemma}
\newtheorem{proposition}[theorem]{Proposition}
\newtheorem{corollary}[theorem]{Corollary}

\theoremstyle{definition}
\newtheorem{definition}[theorem]{Definition}

\theoremstyle{remark}

\numberwithin{equation}{section}

\newcommand{\bs}{\mathbf s}

\newcommand{\eps}{\varepsilon}

\def\Dio{{\rm Dio}}
\def\dio{{\rm dio}} 
\def\ice{{\rm ice}}
\def\rep{{\rm rep}}

\begin{document}

\title[A new complexity function and repetitions in Sturmian words]  
{A new complexity function, repetitions in Sturmian words, and irrationality exponents of Sturmian numbers}   

\author{Yann Bugeaud}
\address{Universit\'e de Strasbourg, CNRS\\ 
IRMA, UMR 7501\\ 
7 rue Ren\'e Descartes\\
67084 Strasbourg, France} 
\email{bugeaud@math.unistra.fr}

\author{Dong Han Kim}
\address{Department of Mathematics Education,
Dongguk University -- Seoul, Seoul 04620, Korea.}
\email{kim2010@dongguk.edu}

\begin{abstract}
We introduce and study a new complexity function in combinatorics on words, 
which takes into account  the smallest second occurrence time   
 of a factor of an infinite word. 
We characterize the eventually periodic words and the
Sturmian words by means of this function. 
Then, we establish a new result on repetitions  
 in Sturmian words and show that 
it is best possible. Let $b \ge 2$ be an integer. 
We deduce a lower bound for the irrationality exponent of real numbers 
whose sequence of $b$-ary digits is a Sturmian sequence over $\{0, 1, \ldots , b-1\}$ and
we prove that this lower bound is best possible. 
As an application, we derive  
some information on the $b$-ary expansion of $\log (1 + \frac{1}{a})$, 
for any integer $a \ge 34$.
\end{abstract}

\subjclass[2010]{68R15 (primary); 11A63, 11J82 (secondary)}

\keywords{Combinatorics on words, Sturmian word, complexity, $b$-ary expansion}

\def\Dio{{\rm Dio}}
\def\dio{{\rm dio}} 
\def\ice{{\rm ice}}
\def\rep{{\rm rep}}

\maketitle 

\section{Introduction}

Let $\mathcal A$ be a finite set called an alphabet and denote by $|\mathcal A|$ its cardinality.
A word over $\mathcal A$ is a finite or infinite sequence of elements of $\mathcal A$.
For a (finite or infinite) word ${\mathbf x} = x_1 x_2 \ldots$ written over $\mathcal A$,
let $n \mapsto p (n,{\mathbf x})$ denote its subword complexity function
which counts the number of different subwords of length $n$ occurring in $\mathbf x$, that is,
$$
p (n,{\mathbf x}) = \# \{ x_k x_{k+1} \dots x_{k+n-1} : k \ge 1 \}, \quad n \ge 1.
$$
Clearly, we have
$$
1 \le p(n,{\mathbf x}) \le |\mathcal A|^n, \quad n \ge 1.
$$
A celebrated theorem by Morse and Hedlund \cite{MoHe} characterizes
the eventually periodic words by means of the subword complexity function.

\begin{theorem}\label{MH}
Let ${\mathbf x} = x_1 x_2 \ldots $ be an infinite word. 
The following statements are equivalent:

(i)  $\mathbf x$ is eventually periodic;

(ii) There exists a positive integer $n$ with $p(n,{\mathbf x}) \le n$;

(iii) There exists $M$ such that $p(n,{\mathbf x}) \le M$ for $n \ge 1$.
\end{theorem}

Therefore, the least possible subword complexity for a
non eventually periodic infinite word $\mathbf x$ is given by 
$p(n,{\mathbf x}) = n+1$ for every $n \ge 1$. 

\begin{definition}
A Sturmian word is an infinite word $\mathbf x$ such that 
$p(n,{\mathbf x}) = n+1$ for every $n \ge 1$. 
\end{definition}

There are uncountably many Sturmian words. There are several ways for describing 
them, one of them is given at the beginning of Section~3.


In the present paper, we introduce and study a new complexity function, 
which takes into account  the smallest second occurrence  
time of a factor of $\mathbf x$. 
For an infinite word ${\mathbf x}= x_1 x_2 \dots $ set 
$$ 
r(n,{\mathbf x}) = \min \{ m \ge 1 :   
x_{i}^{i+n-1} = x_{m-n+1}^{m} \text{ for some } i \text{ with } 1 \le i \le m-n \} .
$$
Here and below, for integers $i, j$ with $i \le j$, we write 
$x_i^j$ for the factor $x_i x_{i+1} \ldots x_j$ of $\mathbf x$. 

Said differently, $r(n,{\mathbf x})$ denotes the length of the smallest prefix of $\mathbf x$
containing two (possibly overlapping) occurrences of some word of length $n$. 

One of the purposes of the present work is to 
characterize the eventually periodic words and the
Sturmian words by means of the function $n \mapsto r(n,{\mathbf x})$. This is the object of 
Theorems~\ref{thm_periodic} and~\ref{thm_sturmian}.

In Section 3, by means of a precise combinatorial study of Sturmian
words, we establish that
every Sturmian word $\bs$ satisfies
$$
\liminf_{n \to + \infty} \, \frac{r(n,\bs)}{n} \leq \sqrt{10} - \frac{3}{2}.  \leqno (1.1)
$$
A similar result also follows from Theorem 2.1 of \cite{BuKi}, 
but with $\sqrt{10} - \frac{3}{2}$ replaced by a larger value strictly less than $2$.     
We prove that the inequality (1.1) is best possible by constructing explicitly a Sturmian 
word $\bs$ for which we have equality in (1.1). 

By Sturmian number, we mean a real number for which there exists an integer 
base $b \ge 2$ such that its $b$-ary expansion is a Sturmian sequence
over $\{0, 1, \ldots , b-1\}$. We show in Section 4 how 
it easily follows from (1.1) that the irrationality exponent of any Sturmian number 
is at least equal to $\frac {5}{3} + \frac{4\sqrt{10}}{15}$.  
We establish that this lower bound is best possible and, more generally,
that the irrationality exponent of any Sturmian number can be read on its 
$b$-ary expansion (which means that infinitely many of its very good rational
approximants can be constructed 
by cutting its $b$-ary expansion and completing by 
periodicity; see below Theorem~\ref{minirr}).

Combined with earlier results of Alladi and Robinson \cite{AlRo80},
our result implies that, for any integer $b \ge 2$, the tail of the $b$-ary
expansion of $\log (1 + \frac{1}{a} )$, viewed as an infinite word over 
$\{0, 1, \ldots , b-1\}$, cannot be a Sturmian word when $a \ge 34$ is 
an integer.

The present paper illustrates the fruitful interplay between combinatorics on words and 
Diophantine approximation, which has already led recently to several progresses.    
It is organized as follows.
Our new results are stated in Sections 2 to 4 and proved in Sections 5 to 8. 
We consider in Section~\ref{sec9} a recurrence function studied by Cassaigne in \cite{Cassa97}. 
The link between the function $n \mapsto r(n,{\mathbf x})$ and other combinatorial exponents is discussed in Section 10.

\section{A new characterization of periodic and Sturmian words}

We begin this section by stating some immediate 
properties of the function $n \mapsto r(n,{\mathbf x})$.

\begin{lemma}\label{prelim}
For an arbitrary infinite word $\mathbf x$ written over a finite alphabet ${\mathcal A}$, we have:

(i) $n+1 \le r(n,{\mathbf x}) \le |\mathcal A|^n + n, \quad (n \ge 1)$.  

(ii)  There exists a unique integer $j$ such that 
$ x_{j}^{j+n-1} = x_{r(n,{\mathbf x})-n+1}^{r(n,{\mathbf x})}$ and $1\le j \le r(n,{\mathbf x}) - n$. 

(iii) $r(n+1,{\mathbf x}) \ge r(n,{\mathbf x}) +1, \quad (n \ge 1)$. 
\end{lemma}

Let $b \ge 2$ and $n \ge 1$ be integers.
A de Bruijn word of order $n$ 
over an alphabet of cardinality $b$ is a word of length $b^n + n - 1$
in which every block of length $n$ occurs exactly once.
Every de Bruijn word of order $n$ over an alphabet with at least three letters
can be extended to a de Bruijn word of order $n+1$ (see e.g. \cite{CW86, Iva87,BeHe11}).  
When $|\mathcal A| \ge 3$, this establishes the 
existence of infinite words $\mathbf x$ satisfying 
$r(n,{\mathbf x}) = |\mathcal A|^n + n$, for every $n \ge 1$. 
Thus, we can have equality in the right hand side of (i)  
for every $n \ge 1$.

The lemma below shows that $r(n,{\mathbf x})$ is bounded from above in terms of 
the subword complexity function of $\mathbf x$.

\begin{lemma}\label{ubound}
For any infinite word $\mathbf x$, we have
$$
r(n,{\mathbf x}) \le p(n,{\mathbf x}) + n, \quad n \ge 1. 
$$
\end{lemma}

\begin{proof}
By the definition of $r(n,{\mathbf x})$,
all the $r(n,{\mathbf x})-1- (n-1)$ factors of length $n$
of $x_1^{r(n,{\mathbf x})-1}$ are distinct. 
Since $x_{r(n,{\mathbf x})-n+1}^{r(n,{\mathbf x})}$ is a factor of $x_1^{r(n,{\mathbf x})-1}$, we have
\begin{equation*}
p(n,{\mathbf x}) \ge p(n,x_1^{r(n,{\mathbf x})-1}) = p(n,x_1^{r(n,{\mathbf x})}) 
=  r(n,{\mathbf x})-n. \qedhere
\end{equation*}
\end{proof}

We stress that there is no analogue upper  
bound for the subword complexity   
function of $\mathbf x$
in terms of $r(n,{\mathbf x})$. 
Indeed, any infinite word ${\mathbf x}= x_1 x_2 \ldots$ over a finite alphabet $\mathcal A$ 
and such that
$$
x_1 \ldots x_{2^j} = x_{2^{j+1} + 2^j +1} \ldots x_{2^{j+2}},   \text{ for } j \ge 1, 
$$
satisfies $r(2^j,{\mathbf x}) \le 2^{j+2}$ for $j \ge 1$, thus 
$r(n,{\mathbf x}) \le 8n$ for every $n \ge 1$. However, by a suitable choice of
$x_{2^j + 1}, \ldots , x_{2^{j+1} + 2^j}$, we can guarantee that
$p(n,{\mathbf x}) = |\mathcal A|^n$ for every $n \ge 1$.

Our first result is a characterization of eventually periodic words by means 
of the function $n \mapsto r(n,{\mathbf x})$. 
It is the analogue of Theorem~1.1.

\begin{theorem}\label{thm_periodic}
Let ${\mathbf x}= x_1 x_2 \ldots $ be an infinite word. 
The following statements are equivalent:

(i)  $\mathbf x$ is eventually periodic;

(ii) $r(n,{\mathbf x}) \le 2n $ for all sufficiently large integers $n$;

(iii) There exists $M$ such that $r(n,{\mathbf x}) -n \le M$ for $n \ge 1$.
\end{theorem}

Our second result is a characterization of Sturmian words by means 
of the function $n \mapsto r(n,{\mathbf x})$.

\begin{theorem}\label{thm_sturmian}
Let ${\mathbf x}= x_1 x_2 \ldots $ be an infinite word. 
The following statements are equivalent:

(i)  $\mathbf x$ is a Sturmian word; 

(ii) $\mathbf x$ satisfies $r(n,{\mathbf x}) \le 2n +1 $ for $n \ge 1$, with equality for
infinitely many $n$. 
\end{theorem}

It is possible to precisely describe the sequence $(r(n,{\mathbf x}))_{n \ge 1}$ for some 
classical infinite words $\mathbf x$, including the Fibonacci word and the Thue-Morse word. 
The proofs of the next results can be obtained by induction. 

Let $\mathbf f$ denote the Fibonacci word ${\mathbf f} = 01001010 \ldots$ 
over $\{0, 1\}$ and $(F_n)_{n \ge 0}$
the Fibonacci sequence given by $F_0 = 0$, $F_1 = 1$ and 
$F_{n+2} = F_{n+1} + F_n$ for $n \ge 0$. 
The Fibonacci word is a Sturmian word and 
it satisfies 
$r(m, {\mathbf f}) = F_n + m$ for $F_n - 2 < m \le F_{n+1} - 2$ and $n \ge 3$.

Let $\mathbf t = 01101001 \dots $ denote the Thue--Morse word over $\{0, 1\}$. Then, we have
$r(1, {\mathbf t}) = 3$ 
and $r(2^n - m, {\mathbf t}) = 5 \cdot 2^{n-1} - m$, if $0 \le m < 2^{n-1}$ and $n \ge 1$.  

\medskip

There are several ways to measure the complexity of an infinite word ${\mathbf x}$, beside 
the functions $n \mapsto p(n,{\mathbf x})$ and $n \mapsto r(n,{\mathbf x})$ already mentioned;
see, for instance, \cite{KK}.
One can also consider the return time function $n \mapsto R(n,{\mathbf x})$, 
which indicates the first return time of the prefix of length $n$ of ${\mathbf x}$. 
The characterization of Sturmian words by means of the 
function $n \mapsto R(n,{\mathbf x})$ is studied in \cite{Kim}.
The main drawback is that $R(\cdot,{\mathbf x})$ is defined only 
when ${\mathbf x}$ is a recurrent word.
Indeed, if ${\mathbf x}$ is an infinite word over a finite alphabet and $a$ is a letter, 
then the fact that $R(n,{\mathbf x})$ is well defined does not imply  
that $R(n,a{\mathbf x})$ is also defined; however, we always have  
$$
r(n-1,{\mathbf x}) + 1 \le r(n,a{\mathbf x}) \le r(n,{\mathbf x}) + 1. 
$$

\section{Combinatorial study of Sturmian and quasi-Sturmian words}

We begin by a classical result on Sturmian words.

\begin{theorem}\label{sturm}
Let $\theta$ and $\rho$ be real numbers with
$0 < \theta < 1$ and $\theta$ irrational.
For $n \ge 1$, set
$$
s_n := \big\lfloor (n+1) \theta + \rho \big\rfloor -\big\lfloor n \theta + \rho \big\rfloor,
\quad s'_n := \big\lceil (n+1) \theta + \rho \big\rceil - \big\lceil n \theta + \rho \big\rceil, 
$$
and define the infinite words
$$
\bs_{\theta, \rho} := s_1 s_2 s_3 \ldots,
\quad
\bs'_{\theta, \rho} := s'_1 s'_2 s'_3 \ldots
$$
Then we have
$$
p(n, \bs_{\theta, \rho}) = p(n, \bs'_{\theta, \rho}) = n+1,
\quad \hbox{for $n \ge 1$}.
$$
The infinite words $\bs_{\theta, \rho}$
and $\bs'_{\theta, \rho}$ are called the
{\rm Sturmian words with slope $\theta$
and intercept $\rho$}.    
Conversely, for every infinite word $\mathbf x$ on $\{0, 1\}$
such that $p(n,{\mathbf x}) = n + 1$ for $n \ge 1$, 
there exist real numbers $\theta_{\mathbf x}$ 
and $\rho_{\mathbf x}$ with
$0 < \theta_{\mathbf x} < 1$ and $\theta_{\mathbf x}$ irrational,
such that ${\mathbf x} = \bs_{\theta_{\mathbf x}, \rho_{\mathbf x}}$
or $\bs'_{\theta_{\mathbf x}, \rho_{\mathbf x}}$.
\end{theorem}

For $\theta$ and $\rho$ as in Theorem~\ref{sturm} 
the words $\bs_{\theta, \rho}$ and $\bs'_{\theta, \rho}$ differ only by at most two letters.  
Classical references on Sturmian words include 
\cite[Chapter~6]{Fogg02}, \cite[Chapter~2]{Loth02}, and \cite[Chapter~9]{AlSh03}.


The function $n \mapsto r(n,{\mathbf x})$ motivates the introduction 
of the exponent of repetition of an infinite word. Although the term `repetition' usually refers to
consecutive copies of the same word, we have decided to use it in our context, 
where we allow overlaps. 

\begin{definition}
The exponent of repetition   
 of an infinite word $\mathbf x$, 
denoted by $\rep({\mathbf x})$, is defined by
$$
\rep ({\mathbf x}) = \liminf_{n \to + \infty} \, \frac{r(n,{\mathbf x})}{n}. 
$$
\end{definition} 

A combinatorial study of Sturmian words whose slope has an
unbounded sequence of partial quotients 
in its continued fraction expansion has been made in Section 11 of \cite{AdBu11}.

\begin{theorem}\label{infiniteslope}
Let $\bs$ be a Sturmian word. If its slope has an unbounded 
sequence of partial quotients in its continued fraction expansion, then
$
\rep(\bs) = 1.
$
\end{theorem}

Theorem \ref{infiniteslope} follows from the proof of  \cite[Proposition 11.1]{AdBu11}.
For the sake of completeness, we provide an alternative 
(in our opinion, simpler) proof in Section~7. 

A result of Berth\'e, Holton, and Zamboni~\cite{BeHoZa06} on the
initial critical exponent (see Definition \ref{defice} below) of Sturmian words
implies straightforwardly that, for every Sturmian word $\bs$, there exists a positive
real number $\delta(\bs)$ such that 
$$
\rep(\bs) \le 2 - \delta(\bs). 
$$
However, the infimum of $\delta(\bs)$ taken over all the Sturmian words $\bs$ is equal to $0$.
The purpose of the next result is to show that the exponents of repetition    
of Sturmian words are uniformly bounded from above by some constant strictly less than $2$.


\begin{theorem} \label{newthm}
Every Sturmian word $\bs$ satisfies
$$
\rep(\bs) \le \sqrt{10} - \frac{3}{2} \  = 1.6622776 \ldots.
$$
Moreover, if a Sturmian word $\bs'$ satisfies
$$
\rep(\bs') = \sqrt{10} - \frac{3}{2}, 
\leqno (3.1) 
$$
then the continued fraction expansion of the slope of $\bs'$
is eventually periodic and of the form $[0; a_1, a_2, \ldots, a_K, \overline{2,1,1}]$
for some integer $K$.
\end{theorem}

It was tempting to conjecture that the upper bound $\sqrt{10} - \frac{3}{2}$ 
in Theorem~\ref{newthm} could be 
replaced by the Golden Ratio $\varphi := \frac{1+\sqrt{5}}{2}$   
(note that the Fibonacci word $\mathbf f$ satisfies
$\rep({\mathbf f}) = \varphi$). 
However, we establish in Section 7 that Theorem~\ref{newthm} is best possible, by giving an  
explicit example of a Sturmian word whose exponent of repetition    
is equal to $\sqrt{10} - \frac 32$. 
For example, the Sturmian word $\bs'$ of slope $\frac{\sqrt{10} - 2}{3} = [0;\overline{2,1,1}]$  
and intercept $\frac{1}{3}$ satisfies (3.1). 
A same kind of example has been already studied by Cassaigne \cite{Cassa97}.  
We discuss Cassaigne's recurrence function $n \mapsto R'(n)$ in Section~\ref{sec9}.  

A more precise result is proved in Section 7. Namely, we establish a necessary and sufficient 
condition on a Sturmian word $\bs'$ ensuring that $\rep(\bs') = \sqrt{10} - \frac{3}{2}$
and give examples of such $\bs'$.  
We also remark that $\sqrt{10} - \frac{3}{2}$ is an isolated point of the set of 
real numbers $\rep(\bs)$, where $\bs$ runs over the Sturmian words.




Actually the conclusion of Theorem \ref{newthm} remains true for a 
slightly larger class of words.

\begin{definition}
A quasi-Sturmian word $\mathbf x$ is an infinite word 
which satisfies
$$
p(n,{\mathbf x}) = n + k, \quad \hbox{for $n \ge n_0$}.
$$
\end{definition}

A structure theorem of Cassaigne \cite{Cassa98} allows us to deduce the
next theorem from Theorem \ref{newthm}.

\begin{theorem} \label{newthmquasi}
Every quasi-Sturmian word $\bs$ satisfies 
$
\rep(\bs) \le  \sqrt{10} - \frac{3}{2}.  
$
\end{theorem}

It can be deduced from Theorem 2.1 of \cite{BuKi} 
that every Sturmian or quasi-Sturmian word $\bs$ satisfies $\rep(\bs) \le 1.83929\dots$. 
The proof of Theorems \ref{newthm} and \ref{newthmquasi} follows a completely different 
approach and yields a significant improvement.

We explain in the next section how Theorem~\ref{newthmquasi} 
allows us to get new results on the $b$-ary expansion
of real numbers whose irrationality exponent is slightly larger than $2$.

\section{Rational approximation of quasi-Sturmian numbers and applications}

Ferenczi and Mauduit \cite{FeMa97} studied the combinatorial 
properties of Sturmian words $\bs$ and showed that, for some positive real number
$\eps$ depending only on $\bs$, they contain infinitely many
$(2 + \eps)$-powers of blocks (that is, a block followed by itself and by its 
beginning of relative length at least $\eps$) occurring not too  
far from the beginning. 
Then, by applying a theorem of Ridout \cite{Rid57} from 
transcendence theory, they deduce that, for any integer $b \ge 2$, 
the tail of the $b$-ary expansion
of an irrational algebraic number, viewed as an infinite word over the alphabet
$\{0, 1, \ldots , b-1\}$, cannot be a Sturmian word; see also \cite{All00}. 

Subsequently, Berth\'e, Holton and Zamboni \cite{BeHoZa06} established
that any Sturmian word $\bs$, whose slope has a bounded continued fraction expansion,
has infinitely many prefixes which are $(2 + \eps)$-powers of blocks, 
for some positive real number
$\eps$ depending only on $\bs$.
This gives non-trivial information on the rational approximation 
to real numbers whose expansion in some integer base is a Sturmian word. 

\begin{definition}
The irrationality exponent $\mu(\xi)$ of a real number $\xi$ is the supremum
of the real numbers $\mu$ such that the inequality
$$
\biggl| \xi - \frac{p}{q} \biggr| < \frac{1}{q^{\mu}}
$$
has infinitely many solutions in rational numbers $\frac{p}{q}$.
If $\mu(\xi)$ is infinite, then $\xi$ is called a Liouville number. 
\end{definition}

Recall that the irrationality exponent of an irrational number $\xi$ is always 
at least equal to $2$, with equality for almost all $\xi$, in the sense of the  
Lebesgue measure. 

As observed in \cite{Ad10} (see also Section~8.5 of \cite{BuLiv2}), 
it follows from the results of \cite{BeHoZa06} and \cite{AdBu11} that, for any integer $b \ge 2$
and for any quasi-Sturmian word $\bs$ over $\{0, 1, \ldots , b-1\}$, there exists
a positive real number $\eta (\bs)$ such that the irrationality
exponent of any real number whose $b$-ary expansion 
coincides with $\bs$ is at least equal to $2 + \eta(\bs)$. 

The reason for this is that, for an integer $b \ge 2$,   
there is a close connection between the exponent of repetition    
of an infinite word $\mathbf x$ written over $\{0, 1, \ldots , b-1\}$ and the
irrationality exponent of the real number whose $b$-ary expansion 
is given by $\mathbf x$. 

\begin{theorem}\label{minmurep}
Let $b \ge 2$ be an integer and ${\mathbf x}= x_1 x_2 \ldots$ an infinite word over
$\{0, 1, \ldots , b-1\}$, which is not eventually periodic. 
Then, the irrationality exponent of the irrational number 
$\xi_{{\bf x}, b} := \sum_{k \ge 1} \, \frac{x_k}{b^k}$ satisfies 
$$
\mu(\xi_{{\bf x}, b}) \ge \frac{\rep({\mathbf x})}{\rep({\mathbf x}) - 1},       \leqno (4.1)
$$
where the right hand side is infinite if $\rep({\mathbf x}) = 1$. 
\end{theorem}

It immediately follows from Theorems \ref{infiniteslope} and \ref{minmurep} that 
any Sturmian number constructed from a Sturmian sequence whose slope has  
unbounded partial quotients is a Liouville number. 
This result was first established by Komatsu \cite{Kom96}.

As mentioned in Section~3 for the related quantity $\delta(\bs)$, 
the infimum of $\eta(\bs)$ over all Sturmian words $\bs$ is equal to $0$
and one cannot deduce a 
non-trivial lower bound for the irrationality exponents of Sturmian numbers. 
We improve this as follows. 

\begin{theorem}\label{minirr}
Let $b \ge 2$ be an integer.
Let $\bs = s_1 s_2 \ldots$ be a Sturmian or a quasi-Sturmian word over
$\{0, 1, \ldots , b-1\}$. Then,  
$$
\mu \Bigl( \, \sum_{j \ge 1} \, \frac{s_j}{b^j} \, \Bigr) 
\ge \frac {5}{3} + \frac{4\sqrt{10}}{15} = 2.5099\ldots ,
$$
with equality when $\bs$ is the Sturmian word $\bs'$ defined in Theorem \ref{newthm}.  
\end{theorem}

The first statement of  
Theorem~\ref{minirr} is an immediate consequence of Theorem~\ref{minmurep} combined
with Theorem~\ref{newthmquasi}. 
Its second statement directly follows from Theorem~\ref{readirrexp} below. 


If there is equality in (4.1), we say that {\it the irrationality exponent of $\xi_{{\bf x}, b}$ 
can be read on   
its $b$-ary expansion}. This is equivalent to say that, for every $\eps > 0$, there exist 
positive integers $r, s$, with $r+s$ being arbitrarily large, such that
$$
\Bigl| \xi_{{\bf x}, b} - \frac{p_{r,s}}{b^r (b^s - 1)} \Bigr| 
\le \frac{1}{b^{(r+s)(\mu (\xi_{{\bf x}, b}) - \eps)}}, 
$$
where $p_{r, s}$ is the nearest integer to $b^r (b^s - 1)  \xi_{{\bf x}, b}$. 
Or, if one prefers, this is equivalent to say that, among the very good 
approximants to $\xi_{{\bf x}, b}$, 
infinitely many of them can be constructed 
by cutting its $b$-ary expansion and completing by 
periodicity (this does not mean, however, that infinitely many convergents to $\xi_{{\bf x}, b}$ have 
a denominator of the form $b^r (b^s - 1)$). 
Using the Diophantine exponent $v'_b$ introduced in \cite{AmBu10} (see
also Section 7.1 of \cite{BuLiv2}), to say that {\it the irrationality exponent of $\xi_{{\bf x}, b}$ 
can be read on its $b$-ary expansion}.
simply means that $v'_b (\xi_{{\bf x}, b}) = \mu (\xi_{{\bf x}, b})$. 

Let $b \ge 2$ be an integer.
A covering argument shows that, for any positive real number $\eps$, the set of real 
numbers $\xi$ such that there are infinitely many integer triples $(r, s, p_{r,s})$  
with $r \ge 0, s \ge 0$ and  
$$
\Bigl| \xi - \frac{p_{r,s}}{b^r (b^s - 1)} \Bigr| 
\le \frac{1}{b^{(r+s)(1 + \eps)}}, 
$$
has Lebesgue measure zero.  
Consequently, the $b$-ary expansion ${\bf x}_{\xi, b}$ of almost 
every real number $\xi$ satisfies $\rep ({\bf x}_{\xi, b}) = + \infty$, thus 
the right-hand side of inequality (4.1) is equal to $1$ almost always. 
This shows that, since the irrationality exponent of an irrational  
number is always at least equal to $2$, 
it can only very rarely be read on its $b$-ary expansion.     
There are only few known examples for which this is the case; 
see \cite{Bu08a,BuKrSh11} and the following result of Adams and Davison \cite{AdDa77}
(additional references and a more detailed statement are given in 
Section 7.6 of \cite{BuLiv2}).

\begin{theorem}\label{characSturm}
Let $b \ge 2$ be an integer and
$\alpha = [a_1; a_2, a_3 \ldots]$ an irrational number greater than $1$.
The irrationality exponent of the real number 
$$
\xi_{\alpha, b} =  \sum_{j = 1}^{+ \infty} \, 
\frac{1}{b^{\lfloor j  \alpha \rfloor}} 
$$
is given by
$$
\mu (\xi_{\alpha, b}) = 1 + \limsup_{n \to + \infty} \,
[a_n ; a_{n-1}, \ldots , a_1].   
$$
\end{theorem}

Theorem \ref{characSturm} gives us the irrationality exponent of any real number
whose expansion in some integer base is a {\it characteristic Sturmian word}   
 (that is, a Sturmian word whose intercept is $0$). 
It shows that equality holds in (4.1) when 
${\bf x}$ is a characteristic Sturmian word. 
We extend this result in Section 8 by proving that 
the inequality in Theorem \ref{minmurep} is an equality for any 
Sturmian word~${\bf x}$ and any integer base $b \ge 2$. 

\begin{theorem}\label{readirrexp}
Let $b \ge 2$ be an integer and ${\mathbf x}= x_1 x_2 \ldots$ a Sturmian word. 
Then, the irrationality exponent of the irrational number  
$\sum_{k \ge 1} \, \frac{x_k}{b^k}$ satisfies 
$$
\mu \Bigl( \sum_{k \ge 1} \, \frac{x_k}{b^k} \Bigr) = \frac{\rep({\mathbf x})}{\rep({\mathbf x}) - 1},
$$
where the right hand side is infinite if $\rep({\mathbf x}) = 1$. 
\end{theorem}

The proof of Theorem \ref{readirrexp} rests on the theory of continued fractions 
combined with properties of the function $n \mapsto r(n, {\bf x})$ and of Sturmian words. 
Furthermore, a result obtained in the course of this proof implies that,  
given $b$ and $b'$ multiplicatively independent integers, 
an irrational real number cannot have simultaneously a Sturmian $b$-ary expansion and a 
Sturmian $b'$-ary expansion. This gives a partial answer to Problem 3 of \cite{Bu12}. 
We will return to this question in a subsequent work.

We display below a statement equivalent to Theorem \ref{minirr}, 
but we need first to introduce some notation. 
Let $b$ denote an integer at least equal to $2$.
Any real number $\xi$ has a unique $b$-ary expansion, that is,
it can be uniquely written as
$$
\xi = \lfloor \xi \rfloor + \sum_{\ell \ge 1} \, 
\frac{a_{\ell}}{b^{\ell}} = \lfloor \xi \rfloor
+ 0 \, . \,  a_1 a_2 \ldots,  
$$
where $\lfloor \cdot \rfloor$ denotes the
integer part function, the {\it digits} $a_1, a_2, \ldots $ are integers
from the set $\{0, 1, \ldots , b-1\}$ and $a_{\ell}$ differs from $b-1$
for infinitely many indices $\ell$. 
A natural way to measure the 
complexity of $\xi$ is to count the number of
distinct blocks of given length 
in the infinite word ${\mathbf a} = a_1a_2 a_3 \ldots$
For $n \ge 1$, we set
$p(n, \xi, b) = p(n, {\mathbf a})$ with ${\mathbf a}$
as above. Clearly, we have
$$
p(n, \xi, b) = \# \{a_{\ell+1} a_{\ell+2} \ldots a_{\ell+n} : \ell \ge 0\}.
$$

\begin{theorem}\label{complsmallmu}
Every irrational real number $\xi$ with
$\mu(\xi) < \frac {5}{3} + \frac{4}{15}\sqrt{10}$ satisfies 
$$
\lim_{n \to + \infty} \, \big( p(n, \xi, b ) - n \big) = + \infty,  
$$
for every integer $b \ge 2$.
Furthermore, for every integer $b \ge 2$, there exists an irrational real number $\xi_b$  
with $\mu(\xi_b) = \frac {5}{3} + \frac{4}{15}\sqrt{10}$ and  
$p(n, \xi_b, b) = n+1$ for $n \ge 1$.   
\end{theorem}

The conclusion of the first assertion of Theorem \ref{complsmallmu} was proved    
to be true for every irrational algebraic number $\xi$ in \cite{FeMa97}, 
for every real number $\xi$ whose irrationality exponent is equal to $2$ in \cite{Ad10} 
(see also Section 8.5 of \cite{BuLiv2}; note that, by Roth's theorem \cite{Roth55}, every irrational
algebraic number satisfies $\mu (\xi) = 2$), 
and for every irrational real number $\xi$ satisfying $\mu(\xi) < 2.19149\dots $ in \cite{BuKi}.

We can deduce from Theorem \ref{complsmallmu} 
some information on the
$b$-ary expansion of several classes of real numbers, 
without knowing exactly their irrationality exponent.     
Recall that, for example, 
Alladi and Robinson \cite{AlRo80} (who improved earlier results of A. Baker \cite{Bak64})
and Danilov \cite{Da78} 
proved that, for any positive integer $s$, the irrationality exponents of
$\log (1 + \frac{s}{t})$ and $\sqrt{t^2 - s^2} \arcsin \frac{s}{t}$ are bounded 
from above by functions of $t$ which tend to $2$ as the integer $t$
tends to infinity.  
The next statement then follows at once from Corollary 1 of \cite{AlRo80}, which 
implies that the irrationality exponent of $\log (1 + \frac{1}{a})$ is less than
$\frac {5}{3} + \frac{4}{15}\sqrt{10}$ for every integer $a \ge 34$.  

\begin{corollary}\label{corlog}
For every integer $b \ge 2$ and every integer $a \ge 34$, we have  
$$
\lim_{n \to + \infty} \,  \left( p \Bigl(n, \log \Bigl(1 + \frac{1}{a} \Bigr), b \Bigr) - n \right) = + \infty,  
$$
\end{corollary}

For much larger values of $a$, a stronger result than the above corollary 
has been established in \cite{BuKi}. Namely, for any positive real number $\eps$, there
exists an integer $a_0$ such that, 
for every integer $b \ge 2$ and every integer $a \ge a_0$, we have
$$
\liminf_{n \to + \infty} \, 
\frac{p \bigl(n, \log \bigl(1 + \frac{1}{a} \bigr), b \bigr)}{n} \ge \frac{9}{8} - \eps. 
$$
The approach followed in \cite{BuKi} gives a non-trivial result
only when the integer $a$ exceeds $23347$. 


\section{Auxiliary combinatorial lemmas}

The proofs of Theorems~\ref{thm_periodic} and~\ref{thm_sturmian} 
rest on a series of combinatorial lemmas.

For a word  $U = u_1 \dots u_n$ composed of $n$ letters, denote by $|U|= n$ its length 
and set    
$$ 
\Lambda(U) = \{ 1 \le k < n :  u_i = u_{i + k} \text { for all }1 \le i \le n-k \}.    
$$
An element of $\Lambda(U)$ is called a period of $U$. 
We stress that a period of a word of length $n$ may not be a divisor of $n$. 
A finite word $U$ is called {\it primitive}   
 if there is no non-empty word $V$ such that 
$U = V^n$ for some integer $n \ge 2$. 

\begin{lemma}[Fine and Wilf Theorem \cite{FW}]\label{fl}   
Let $U = u_1 \dots u_n$  
 and $h, k$ be in $\Lambda (U)$.  
If $n\ge h + k - \textrm{gcd}(h,k)$, then $U$ is periodic of period $\textrm{gcd}(h,k)$.  
\end{lemma}

\begin{lemma}\label{forbidden}
Let $U = u_1 \dots u_n$ be a finite word and $\lambda$ in $\Lambda(U)$.
Then $u_{n-\lambda+2}^{n} a$ with $a \ne u_{n-\lambda+1}$ is not a factor of $U$. 
\end{lemma}

\begin{proof}
Since $\lambda$ is in $\Lambda(U)$, all the factors of length $\lambda$ in $U$ 
have the same number of $a$'s. 
Since $u_{n-\lambda+1} \ne a$, 
the number of $a$'s in $u_{n-\lambda+1}^n$ is one less than in $u_{n-\lambda+2}^{n} a$, 
thus the latter cannot be a factor of $U$.
\end{proof}

\begin{lemma}\label{jump}
Let $\mathbf x$ be an infinite word and $n$ an integer with 
$r(n,{\mathbf x}) \ge r(n-1,{\mathbf x}) + 2$.
Then $r(n,{\mathbf x}) \ge 2n+1.$
\end{lemma}

\begin{proof}
To shorten the notation, we simply write $r( \cdot)$ for $r(\cdot ,{\mathbf x})$.
Suppose that 
$$ 
r(n) \ge r(n-1) + 2 \ \text{ but } \ r(n) \le 2n.
$$ 
Let $s, \ell$ be the nonnegative integers satisfying 
\begin{equation}\label{bound1}
x_{s+1}^{s+n-1} = x_{r(n-1)-n+2}^{r(n-1)}, \qquad x_{r(n)-n+1}^{r(n)}= x_{r(n)-n+1-\ell}^{r(n)-\ell}
\end{equation}
with 
\begin{equation}\label{bound2}
0 \le s \le r(n-1) - n, \qquad 1 \le \ell \le  r(n)-n \le n.
\end{equation}
Then, we have 
\begin{equation}\label{6.6}
x_{s+n} \ne x_{r(n-1)+1},
\end{equation}
for otherwise $r(n) = r(n-1) +1$.

Since
$$
r(n-1)-n-s+1 \in \Lambda \left( x_{s+1}^{r(n-1)} \right),
$$
by Lemma~\ref{forbidden} and \eqref{6.6}, the word 
$x_{n+s+1}^{r(n-1)+1} $ is not a factor of $x_{s+1}^{r(n-1)}$.

Our assumption implies $n + s+ 1 \ge r(n) - n +1$ and $r(n-1)+1 \le r(n)$, thus 
by \eqref{bound1}, we have 
$x_{n+s+1}^{r(n-1)+1} = x_{n+s+1-\ell}^{r(n-1)+1-\ell} $, 
which is not a factor of $x_{s+1}^{r(n-1)}$.
Therefore, we have $n+s+1- \ell < s+1$, i.e., $n < \ell $, a contradiction to \eqref{bound2}.
\end{proof}

\begin{lemma}\label{lem_per}
Let $\mathbf x$ be an infinite word and $n$ an integer such that
$r(n+1,{\mathbf x}) = r(n,{\mathbf x}) + 1$. 
Let $j$ be the integer satisfying $1\le j < r(n,{\mathbf x}) - n +1$
and $x_{j}^{j+n-1} = x_{r(n,{\mathbf x})-n+1}^{r(n,{\mathbf x})}$.
Then, $x_{j+n} = x_{r(n,{\mathbf x})+1}$.
\end{lemma}

\begin{proof}
By assumption, there exists a unique integer $h$ satisfying $1\le h < r(n+1,{\mathbf x}) - n$
and $x_{h}^{h+n} = x_{r(n+1,{\mathbf x})-n}^{r(n+1,{\mathbf x})}$. In particular, we have
$x_{h}^{h+n-1} = x_{r(n+1,{\mathbf x})-n}^{r(n+1,{\mathbf x})-1}$, 
thus $h=j$ and $x_{j+n} = x_{r(n,{\mathbf x})+1}$. 
\end{proof}

\begin{lemma}\label{lem_sturmian}
Let $\mathbf x$ be an infinite word satisfying $r(i,{\mathbf x}) \le 2i+1$ for all $i \ge 1$.
Let $m, n$ be positive integers such that $r(n,{\mathbf x}) = 2n+1$ and $m \ge 2n+1$.
If $k$ is the integer defined by $r(k-1,{\mathbf x}) < m \le r(k,{\mathbf x})$, then 
$k \ge n$ and $r(k,{\mathbf x})-k \le m-n$.
\end{lemma}

\begin{proof}
Write $r(\cdot)$ for $r(\cdot,{\mathbf x})$. 
Observe that $k \ge n$, since $r(n-1) < r(n) \le m$. 
If $r(k) = m$, then we get
$r(k) - k = m-k \le m-n$, as required.

If $r(k-1) < m < r(k)$, then $r(k) \ge r(k-1) + 2$  
and we deduce from 
Lemma~\ref{jump} that $r(k) = 2k +1$. 
Furthermore, we have $k \ge n+1$.   
Let $\ell = \min\{ i \ge 1 : r(k-i) = 2(k-i)+1\}$.  
Since $r(n) = 2n+1$, the integer $\ell$ is well-defined and  
$$
k \ge n + \ell. 
$$
For $i = 1, \ldots , \ell-1$, we have $r(k-i) \le 2(k-i)$ and it follows from
Lemma~\ref{jump} that $r(k-i) = r(k-i-1) + 1$, thus, 
$$
r(k-1) - r(k-\ell) =  \ell -1.
$$
Since $m \ge r(k-1) +1 = r(k-\ell)+\ell = r(k) - \ell$, we have
$$ 
r(k) - k \le (m+\ell) - (n+\ell) = m-n,
$$
which completes the proof of the lemma. 
\end{proof}

\section{Proofs of Theorems~\ref{thm_periodic} and~\ref{thm_sturmian}}

{\it Proof of Theorem~\ref{thm_periodic}.}

(iii) $\Rightarrow$ (ii) : 
Immediate. 

(ii) $\Rightarrow$ (i) : 
It follows from Lemma~\ref{jump} and Lemma~\ref{prelim} (iii) that there exists an integer $n_0$
such that $r(n+1,{\mathbf x}) = r(n,{\mathbf x}) + 1$ for every $n \ge n_0$. 
By Lemma~\ref{lem_per}, we deduce that there exists an integer $j$ such that
$x_{j+n} = x_{r(n_0,{\mathbf x}) + n - n_0 +1}$, for $n \ge n_0$. This shows that 
$\mathbf x$ is eventually periodic. 

(i) $\Rightarrow$ (iii) : 
Let $r$ and $s$ denote the length of the preperiod and that 
of the period of $\mathbf x$. Then, the infinite word starting at $x_{r+1}$
is the same as the infinite word starting at $x_{r+s+1}$, thus we
have $r(n,{\mathbf x}) \le n + r + s$ for $n \ge 1$. 

\medskip

{\it Proof of Theorem~\ref{thm_sturmian}.}

(i) $\Rightarrow$ (ii) : 
The inequality is clear by Lemma~\ref{ubound} and  
Theorem~\ref{thm_periodic} implies that there is equality for infinitely many $n$. 

(ii) $\Rightarrow$ (i) : 
Let $n$ be an integer such that $r(n,{\mathbf x}) = 2n + 1$. 
By the proof of Lemma~\ref{ubound} 
we have $p(n, x_1^{2 n}) = n +1$.

Let $m$ be an integer with $m \ge 2n + 1$.
Then, by Lemma~\ref{lem_sturmian}, there exists an integer $k$ such that 
$k \ge n$, $m \le r(k,{\mathbf x})$ and $r(k,{\mathbf x}) - k \le m-n$.
By Lemma~\ref{prelim} (ii), we get that 
$x_{r(k,{\mathbf x})-k+1}^{r(k,{\mathbf x})} = x_{r(k,{\mathbf x})-k+1 -j}^{r(k,{\mathbf x})-j}$ for some 
integer $j$ with $1 \le j \le r(k,{\mathbf x})-k$.
Therefore, we have $x_{m-n+1}^m = x_{m-n+1-j}^{m-j}$, 
which implies that 
$$ 
p(n,x_1^m) = p(n,x_1^{m-1}). $$  
Since this equality holds for every $m \ge 2n+1$ and $p(n,x_1^{2n}) = n +1$, 
we deduce that $p(n,{\mathbf x}) = n +1$. 
Thus, we have established the existence of arbitrary large integers
$n$ such that $p(n,{\mathbf x}) = n+1$. This shows that $\mathbf x$ is a Sturmian word.

\section{Proof of Theorems~\ref{newthm} and~\ref{newthmquasi}}

Through this section, we fix an infinite sequence $(a_k)_{k\ge 1}$ 
of positive integers.     
We define inductively a sequence of words $(M_k)_{k \ge 0}$ on the 
two letter-alphabet  
 $\{0, 1\}$ by the formulas   
\begin{equation}\label{mk_con}
M_0 = 0, \quad M_1 = 0^{a_1 - 1}  1  \text{ and }   
M_{k+1} = M_k^{a_{k+1}} \, M_{k-1} \qquad (k \ge 1).   
\end{equation}
It is easy to check that the last two letters of $M_k$ are 10 (resp. 01) 
if $k$ is even (resp. odd) and $|M_k| \ge 2$.   
This sequence converges to the infinite word
$$
\bs_{\theta, 0} := \lim_{k \rightarrow + \infty} \,  M_k = 0^{a_1 - 1} 1  \ldots,
$$
which is usually called the characteristic Sturmian word of slope  
$$
\theta := [0; a_1, a_2, a_3, \ldots]
$$
constructed over the alphabet $\{ 0, 1 \}$ (See e.g. \cite{Loth02}). 

Let $\mathbf x$ be a Sturmian word of slope $\theta$. We study the
combinatorial properties of $\mathbf x$. 
An {\it admissible word} is a factor of $\mathbf x$ of finite length. 
Note that the set of factors of $\mathbf x$ is the same as that of $\bs_{\theta, 0}$   
(see e.g.  \cite[Proposition 2.1.18]{Loth02}). 
Let $(\frac{p_\ell}{q_\ell})_{\ell \ge 0}$    
denote the sequence of convergents to the slope    
of  ${\mathbf x}$.  Then, for $k \ge 0$,   
we have $q_k = |M_k|$ and $p_k$ is the number of digits $1$ in $M_k$. 
It is known that only the last two letters of $M_{k+1} M_{k}$ and $M_{k} M_{k+1}$ are different 
(see e.g. \cite[Proposition 2.2.2]{Loth02}).  
For a non-empty finite word $U$, we write $U^-$ for the word $U$ deprived of 
its last letter.     
For $k \ge 1$, set 
$$
\tilde M_k = (M_k M_{k-1})^{--} = (M_{k-1} M_k )^{--}    
$$
and observe that $\tilde M_k$ is a prefix of $M_{k+1}$.

We will use the property that $M_{k+1}M_k$ and $M_{k+1} M_{k+1}M_k$ 
are primitive (see e.g. \cite[Proposition 2.2.3]{Loth02}) in 
conjunction  
with the following lemma.

\begin{lemma}\label{factors}  
Let $U$ be a primitive word.   
Then all the $|U|$ factors of length $|U|-1$ of $UU^{--}$ are distinct.   
\end{lemma}

\begin{proof}  
Assume that there are integers $i, j$ with $0 \le i < j \le |U|-1$ and    
$$ 
(UU^{--})_{i+1}^{i+|U|-1} = (UU^{--})_{j+1}^{j+|U|-1}.  
$$    
Then, $j-i$ and $|U|$ are periods of $(UU^{--})_{i+1}^{j+|U|-1}$   
and     
$$ 
|U| + (j-i) - \textrm{gcd}(|U|,j-i)  \le |U| + j-i - 1.
$$   
Thus, we deduce from Lemma~\ref{fl} that      
$( UU^{--} )_{i+1}^{j+|U|-1}$ is periodic of period $\textrm{gcd}(|U|,j-i)$.   
Since     
$$
\textrm{gcd}(|U|,j-i) \le j-i \le |U| -1,     
$$
this contradicts the fact that $U$ is primitive.   
\end{proof}

The next lemma shows that repetitions   
occur near the beginning of any Sturmian word of slope $\theta$.  

\begin{lemma}\label{cases}
Let $\mathbf x$ be a Sturmian word of slope $\theta$. 
Then, for $k \ge 1$,  
there exists a unique word $W_k$ satisfying 

(i) ${\mathbf x}= W_k M_k \tilde M_k \dots$, where $W_k$ is a non-empty suffix of $M_k$,

\noindent or

(ii) ${\mathbf x}= W_k M_{k-1} M_k \tilde M_k \dots$, where $W_k$ is a non-empty suffix of $M_k$, 

\noindent or 

(iii)  ${\mathbf x}= W_k M_k \tilde M_k \dots$, where $W_k$ is a non-empty suffix of $M_{k-1}$,

\noindent and all the $(2q_k+q_{k-1})$ cases are mutually exclusive. 

Furthermore, if ${\mathbf x}= W_k M_{k-1} M_k \tilde M_k \dots$ 
and $W_k$ is a non-empty   
suffix of $M_k$,  then $W_{k+1} = W_k M_{k-1}$. 
Moreover, if ${\mathbf x}= W_k M_k \tilde M_k \dots$ and $W_k$ is a non-empty suffix  
of $M_{k-1}$, then  $W_{k+1} = W_k$.  
\end{lemma}

\begin{proof}
We first claim that, for each $k \ge 1$, the word   
$M_k M_k M_{k-1} M_k \tilde M_k$ is  admissible.  
This follows from the fact that $M_{k+3} M_{k+2}$ is admissible and 
\begin{align*}  
M_{k+3} &= \cdots M_{k+2} M_{k+1} = \cdots M_k M_{k+1} = \cdots  M_k M_k M_{k-1},  \\  
M_{k+2} &= M_{k+1} M_k \cdots = M_k \tilde M_k  \cdots . 
\end{align*}

Since $M_k M_k M_{k-1}$ is primitive,  
Lemma~\ref{factors} implies that any admissible word  
of length $2q_{k}+q_{k-1}-1$ 
is a factor of $M_k M_k M_{k-1} M_k \tilde M_k$.    
These admissible words are prefixes of $W M_k \tilde M_k$ or $W M_{k-1} M_k \tilde M_k$  
for some non-empty $W$ which is a suffix of $M_k$,   
and prefixes of $W M_k \tilde M_k$ for some non-empty $W$ which is a suffix of $M_{k-1}$.  
Consequently, $\mathbf x = W M_k^- \dots$ or $W M_{k-1} M_k^- \dots$   
with $W$ which is a suffix of $M_k$ or  $\mathbf x = W M_k^- \dots$ 
with $W$ which is a suffix of or $M_{k-1}$ .  

Since there are two admissible words of length $2q_{k}+q_{k-1}-1$ starting with $M_k^-$,  
namely $M_k M_{k-1} M_k^{-}$ and $M_k M_k M_{k-1}^-$,  
it follows that if $\mathbf x = UM_k^- \dots $ for some $U$, 
then $\mathbf x = UM_k \tilde M_k \dots $.   
Hence we conclude  that  
$\mathbf x = W M_k \tilde M_k  \dots$ or $W M_{k-1} M_k \tilde M_k  \dots$   
with $W$ which is a suffix of $M_k$ or  $\mathbf x = W M_k \tilde M_k  \dots$ with $W$ which is a suffix of or $M_{k-1}$.  
Putting $W_k = W$, we see that $W_k$ satisfies one of the cases (i), (ii), (iii), 
which are mutually exclusive by Lemma~\ref{factors}.


By the first assertion of the lemma,   
$\mathbf x$ starts with $W_{k+1} \tilde M_{k+1}$, 
where $W_{k+1}$ is a non-empty suffix of $M_{k+1}$ or $M_k$.    
If $W_{k+1}$ is a suffix of $M_k$, then  put $W' = W_{k+1}$,     
thus 
\begin{equation*}
\mathbf x = W_{k+1} \tilde M_{k+1} \dots = W' M_k \tilde M_{k} \dots.
\end{equation*}   
If $W_{k+1}$ is a suffix of $M_{k+1}= M_k^{a_{k+1} } M_{k-1}$, then   
$W_{k+1} = W' M_k^{t } M_{k-1}$ for some integer $t \ge 0$ and 
a non-empty suffix $W'$ of $M_k$ 
or $W_{k+1}$ is a non-empty suffix of $M_{k-1}$.  
If $W_{k+1} = W' M_k^{t } M_{k-1}$, with $W'$ a suffix of $M_k$,    
then    
\begin{equation*}
\mathbf x = W_{k+1} \tilde M_{k+1} \dots = 
\begin{cases} W' M_k \tilde M_{k} \dots,  &\text{ if } t \ge 1, \\     
W' M_{k-1} M_k \tilde M_{k} \dots,  &\text{ if } t =0.  
\end{cases} \end{equation*}   
If $W_{k+1}$ is a suffix of $M_{k-1}$, then put $W' = W_{k+1}$,   
thus
\begin{equation*}
\mathbf x = W_{k+1} \tilde M_{k+1} \dots = W' M_k \tilde M_{k} \dots.   
\end{equation*}  

By the first assertion of the lemma, we conclude that $W' = W_k$. 
If ${\mathbf x}= W' M_{k-1} M_k \tilde M_k \dots$,  then $W_{k+1} = W' M_{k-1}$   
and if $W'$ is a suffix of $M_{k-1}$, then  $W_{k+1} = W'$.    
\end{proof}

We are now in position to establish Theorems \ref{infiniteslope} and \ref{newthm}.

\noindent {\it Proof of Theorem~\ref{infiniteslope}.}

Let $k$ and $t$ be large integers such that $M_k = (M_{k-1})^t M_{k-2}$.
Let $\ell$ be the integer part of $\sqrt{t}$. 
We distinguish two cases. 
If $|W_k| > (\ell + 1)  |M_{k-1}|$, 
then 
$$
r((\ell-1) |M_{k-1}|, {\mathbf x}) \le \ell   |M_{k-1}|
$$
and, otherwise, we check that
$$
r((t - 1) |M_{k-1}|, {\mathbf x}) \le |W_k| + t  |M_{k-1}| \le (t + \ell + 1) |M_{k-1}|. 
$$
As $k$ and $t$ can be taken arbitrarily large, we deduce that $\rep({\mathbf x}) = 1$. 

\bigskip

\noindent {\it Further auxiliary results for the proof of Theorem \ref{newthm}}.

\begin{lemma}\label{mkfactor}
If ${\mathbf x} = U V \dots$ where $V$ is a factor of $M_k \tilde M_{k+1}$ such that $|V| > q_k$, then we have
$$
r(|V| - q_k, {\mathbf x}) \le |UV|.
$$
\end{lemma}

\begin{proof}
Let $V = v_1 \ldots v_n$ be a factor of $M_k \tilde M_{k+1}$ such that $|V| =n > q_k$.  
Since $M_k \tilde M_{k+1} = M_k \ldots M_k M_{k-1}^{--}$
and $M_{k-1}$ is a prefix of $M_k$, 
we get $v_1^{n-q_k} = v_{1+q_k}^{n}$. 
Thus we have $r(n - q_k, {\mathbf x}) \le |UV|$.  
\end{proof}

\medskip

We establish two further lemmas on the combinatorial structure of Sturmian words. 
For $k \ge 1$, we set    
$$ 
\eta_k := \frac{q_{k-1}}{q_k}, \qquad t_k := \frac{|W_k|}{q_k}, \qquad \eps_k := \frac{2}{q_k}.
$$
Recall that $\varphi$ denotes the Golden Ratio $\frac{1 + \sqrt{5}}{2}$.    

In the rest of the proof of the theorem,  we assume that $k$ is large enough to 
ensure that $q_{k-2} \ge 6$,  
thus, $\eps_k < \eta_k$, $\eps_k < \frac 16$ and $\eps_k < \frac{1- \eta_k}{2}$.

\begin{lemma}\label{lem75}
(i) If ${\mathbf x}= W_k M_k \tilde M_k \dots$, where $W_k$ is a suffix of $M_k$, then
$\frac{r(n, {\mathbf x})}{ n} < \varphi + 2\eps_k$ 
for some $n$ with $q_k -2 \le n \le |W_k| + q_k + q_{k-1} -2$.

(ii) If ${\mathbf x}= W_k M_k \tilde M_k \dots$, where $W_k$ is a suffix of $M_{k-1}$, then 
$\frac{r(n, {\mathbf x})}{n} < \varphi  + 2\eps_k$  
for some $n$  
with $|W_k| + q_k -2 \le n \le |W_k| + q_{k+1} + q_k -2$.
\end{lemma}

\begin{proof}
(i)
Since $W_k M_k \tilde M_k$ is a factor of $M_k \tilde M_{k+1} = M_k M_k M_{k+1}^{--}$,
we have by Lemma~\ref{mkfactor} 
$ r(|W_k M_k \tilde M_k|- q_k, {\mathbf x}) \le |W_k M_k \tilde M_k|,$
which yields that
\begin{equation}\label{l751}
\begin{split}
\frac{r(|W_k|+q_k+q_{k-1}-2, {\mathbf x})}{|W_k|+q_k+q_{k-1}-2} 
&\le \frac{|W_k| + 2q_k + q_{k-1} -2}{|W_k|+q_k+q_{k-1}-2} \\
&= 1 + \frac{1}{t_k+1 +\eta_k-\varepsilon_k} <  1 + \frac{1}{t_k+1 +\eta_k} + \varepsilon_k.  
\end{split}
\end{equation}

Furthermore, ${\mathbf x}= W_k M_k \tilde M_k \dots = W_k \tilde M_k \dots$,
thus we have by Lemma~\ref{mkfactor} 
$ r(|\tilde M_k|- q_{k-1}, {\mathbf x}) \le |W_k \tilde M_k|,$
which yields that
\begin{equation}\label{l752}
\frac{r(q_k -2, {\mathbf x})}{q_k-2} \le \frac{|W_k| + q_k + q_{k-1} -2}{q_k -2} 
= 1 + \frac{t_k +\eta_k }{1- \varepsilon_k} <  1 + t_k +\eta_k + 2\varepsilon_k.   
\end{equation}

Since for every positive real number $x$ we have $\min(x,\frac1{1+x}) \le \frac{1}{\varphi}$, 
we derive from \eqref{l751} and \eqref{l752} that
$$
\frac{r(n, {\mathbf x})}{n} < \varphi +2\varepsilon_k  
\ \text{ for some $n$ with } \ q_k -2 \le n \le |W_k| + q_k + q_{k-1} -2.
$$

(ii)
Since ${\mathbf x}= W_k M_k \tilde M_k \dots = W_k \tilde M_k \dots$ and $W_k \tilde M_k$ is a factor of $M_{k-1} \tilde M_k$, by Lemma~\ref{mkfactor} we have
$ r(| W_k \tilde M_k|- q_{k-1}, {\mathbf x}) \le |W_k \tilde M_k|,$
which yields that
\begin{equation}\label{l753}
\begin{split}
\frac{r(|W_k|+q_k -2, {\mathbf x})}{|W_k|+q_k-2} &\le \frac{|W_k| + q_k + q_{k-1} -2}{|W_k|+q_k -2} \\
&= 1 + \frac{\eta_k}{t_k +1-\varepsilon_k} < 1 + \frac{\eta_k}{t_k +1}+ \varepsilon_k.  
\end{split}
\end{equation}

Since $W_k$ is a suffix of $M_{k-1}$ which is a suffix of $M_{k+1}$,
we deduce from Lemma~\ref{cases} that
${\mathbf x}$ starts with either $W_{k+1} M_{k+1} \tilde M_{k+1}$ or $W_{k+1} M_k M_{k+1} \tilde M_{k+1}$, where $W_{k+1} = W_k$.
If ${\mathbf x} = W_{k+1} M_{k+1} \tilde M_{k+1} \dots$, then the proof is completed by (i)
since $q_{k+1} \ge |W_k|  + q_k$ and $|W_{k+1}| = |W_k|$.   
If ${\mathbf x}= W_{k+1} M_k M_{k+1} \tilde M_{k+1} \dots = W_{k+1} M_k \tilde M_{k+1} \dots$, then by Lemma~\ref{mkfactor} we obtain
$$ r( |M_k \tilde M_{k+1}| - q_k, {\mathbf x}) \le |W_{k+1} M_k \tilde M_{k+1}|,$$
thus,
\begin{align*}
\frac{r( q_{k+1}+q_k -2, {\mathbf x})}{q_{k+1}+q_k -2} 
&\le \frac{|W_{k+1}|+q_{k+1}+2q_k-2}{q_{k+1}+q_k-2}  
= 1+ \frac{|W_{k+1}|+q_k}{q_{k+1}+q_k-2} \\ 
&\le 1+\frac{|W_k|+q_k}{2q_k+q_{k-1}-2} = 1 + \frac{t_k + 1}{2 + \eta_k - \varepsilon_k}  
< 1 + \frac{t_k + 1}{2 + \eta_k} + \varepsilon_k.
\end{align*}
Combined with \eqref{l753}, we deduce that there exists an integer $n$ 
with $|W_k| + q_k -2 \le n \le q_{k+1} + q_k -2$ and 
$$
\frac{r(n,{\mathbf x})}{n} \le 1 + \sqrt{\frac{\eta_k}{2+\eta_k}} + \varepsilon_k 
< 1 + \frac{1}{\sqrt 3} + \varepsilon_k = 1.57735\ldots + \varepsilon_k.  
$$
This completes the proof of the lemma. 
\end{proof}

\begin{lemma}\label{lem76}
Assume that ${\mathbf x}= W_k M_{k-1} M_k \tilde M_k \dots$,   
where $W_k$ is a suffix of $M_k$ and $a_{k} \ge 3$. 
If $k$ is sufficiently large, then, 
for some integer $n$ with $\frac{q_k }{2} -2 \le n \le q_k+q_{k-1}-2$,  
we have 
$$
\frac{r(n, {\mathbf x})}{n} < \frac{\sqrt{17}+9}{8}  + 2\eps_k = 1.640\ldots  + 2\eps_k.  
$$  
\end{lemma}

\begin{proof}
By the assumption $a_{k} \ge 3$, we get $\eta_k = \frac{q_{k-1}}{q_k} < \frac 13$.   

Since ${\mathbf x}= W_k M_{k-1} M_k \tilde M_k \ldots = W_k M_{k-1} \tilde M_k \dots$,  
it follows from Lemma~\ref{mkfactor} that 
$ r(|M_{k-1} \tilde M_k|- q_{k-1}, {\mathbf x}) \le |W_k M_{k-1} \tilde M_k|,$
which yields that
\begin{equation}\label{l761}
\begin{split}
\frac{r(q_k+q_{k-1}-2, {\mathbf x})}{q_k+q_{k-1}-2} &\le \frac{|W_k| + q_k + 2q_{k-1} -2}{q_k+q_{k-1} -2} \\
&=  1 + \frac{t_k +\eta_k}{1+\eta_k - \varepsilon_k}  
< 1 + \frac{t_k +\eta_k}{1+\eta_k} + \varepsilon_k.   
\end{split}
\end{equation}

We also have that ${\mathbf x}= W_k M_{k-1} \tilde M_k \ldots = W_k M_{k-1}^{--} \dots$.
Assume that $|W_k| \ge \frac{q_k}{2} \ge 3$.   
Since $W_k M_{k-1}^{--}$ is a suffix of $M_k M_{k-1}^{--} = \tilde M_k$, 
by Lemma~\ref{mkfactor},  
$ r(|W_k M_{k-1}^{--}|- q_{k-1}, {\mathbf x}) \le |W_k M_{k-1}^{--}|,$
thus
\begin{equation}\label{l762}
\begin{split}
\frac{r(|W_k|-2, {\mathbf x})}{|W_k|-2} &\le \frac{|W_k| + q_{k-1} -2}{|W_k|-2}
 = 1 + \frac{\eta_k}{t_k - \eps_k}  
= 1 +  \frac{\eta_k}{t_k} +\frac{\eta_k\eps_k}{t_k(t_k - \eps_k)}  \\
&< 1 +  \frac{\eta_k}{t_k} +\frac{4\eps_k}{3(1 - 2\eps_k)} 
<  1 + \frac{\eta_k}{t_k} + 2\eps_k.    
\end{split}
\end{equation}
By \eqref{l761} and \eqref{l762}, we get
\begin{equation*}
\min_{\frac{q_k }{2}-2 \le n \le q_k+q_{k-1}-2} \frac{r(n, {\mathbf x})}{n}  <   
\begin{cases}
1 + \min \left\{ \frac{t_k +\eta_k}{1+\eta_k},  \frac{\eta_k}{t_k} \right\} + 2\varepsilon_k, 
&\text{ if } |W_k| \ge \frac{q_k }{2}. \\
1 + \frac{1/2 +\eta_k}{1+\eta_k} + \varepsilon_k, &\text{ if } |W_k| < \frac{q_k }{2}. 
\end{cases}
\end{equation*}  
Since $\min \left\{ \frac{t_k +\eta_k}{1+\eta_k},  \frac{\eta_k}{t_k} \right\} \le  \frac{\eta_k + \sqrt{ 5\eta_k^2 + 4\eta_k}}{2(1+\eta_k)}$, we get 
\begin{equation*}
\min_{\frac{q_k }{2} -2 \le n \le q_k+q_{k-1}-2} \frac{r(n, {\mathbf x})}{n} < 
1 + \frac{\max \left\{ \eta_k + \sqrt{ 5\eta_k^2 + 4\eta_k}, 1 + 2\eta_k  \right\}}{2(1+\eta_k)}
 + 2\varepsilon_k.  
\end{equation*}
Thus, using $\eta_k < \frac 13$,  
for some integer $n$ with
$\frac{q_k }{2}-2 \le n \le q_k+q_{k-1}-2$ we have
\begin{equation*}
\frac{r(n,{\mathbf x})}{n} <  
 1 + \frac{\frac 13 + \sqrt{ \frac 59 + \frac 43}}{2(1+\frac 13)}  + 2\eps_k  
 = \frac{\sqrt{17}+9}{8}  + 2\eps_k . \qedhere   
\end{equation*}
\end{proof}

\bigskip
\goodbreak

\noindent {\it Completion of the proof of Theorem~\ref{newthm}.}

\bigskip

Suppose that  $\liminf_{n \to + \infty}  \frac{r(n,{\mathbf x})}{n} >  1.65$.   
By Lemmas \ref{cases}, \ref{lem75} and \ref{lem76}, for all large $k$ we have  
$a_k \in \{1,2\}$ and 
$${\mathbf x} = W_k M_{k-1} M_k \tilde M_k \dots ,$$ where $W_k$ is a suffix of $M_k$. 
Thus, for all large $k$ we have $W_{k+1} = W_k M_{k-1}$ from Lemma~\ref{cases}.   

We gather two auxiliary statements in a lemma. 

\begin{lemma}\label{lemt1}
Assume that ${\mathbf x}= W_k M_{k-1} M_k \tilde M_k \dots$, 
where $W_k$ is suffix of $M_k$.
If $k$ is sufficiently large, then we have 
\begin{align}
\frac{r(|W_k|+q_k+q_{k-1}-2,{\mathbf x})}{|W_k|+q_k+q_{k-1}-2} &<1 + \frac{1+\eta_{k}}{t_k + 1+\eta_{k}} + \eps_k, \tag{i}\\
\frac{r(q_{k}+q_{k-1}-2,{\mathbf x})}{q_{k}+q_{k-1}-2} &< 1 + \frac{t_k+ \eta_k}{1 + \eta_k} + \eps_k. \tag{ii}
\end{align}
\end{lemma}

\begin{proof} 
Since ${\mathbf x}_{1}^{|W_k|+q_k+q_{k-1}-2} = W_k \tilde M_k = 
{\mathbf x}_{q_k+q_{k-1}+1}^{|W_k|+2q_k+2q_{k-1}-2}$,  we get
$$r(|W_k|+q_k+q_{k-1}-2,{\mathbf x}) \le |W_k| + 2q_k+2q_{k-1}-2.$$
Also by Lemma~\ref{mkfactor},  from the fact ${\mathbf x}=  W_k M_{k-1} M_k \tilde M_k \ldots = W_k M_{k-1} \tilde M_{k} \dots$ 
 we get 
\begin{equation*}
r(q_{k}+q_{k-1}-2,{\mathbf x}) \le |W_{k}|+ q_{k} + 2q_{k-1} -2. \qedhere
\end{equation*}
\end{proof}

$\bullet$ If $a_{k} = 1$ for all large $k$ then 
$\eta_k$ tends to $\frac{1}{\varphi}$ as $k$ tends to infinity and we deduce from
$$
1- t_{k+1} = 1 - \frac{t_k + \eta_k}{1+\eta_k} = \frac{1- t_k}{1+\eta_k}
$$
that $\lim_{k \to + \infty} t_k = 1$. 
By Lemma~\ref{lemt1} (i), we then get
\begin{equation*}
\frac{r(|W_k|+q_k+q_{k-1}-2,{\mathbf x})}{|W_k|+q_k+q_{k-1}-2}  < 1 + \frac{1+\eta_k}{t_k + 1 + \eta_k } + \eps_k, 
\end{equation*}
where the right hand side tends to $1 + \frac{1+1/\varphi}{2+1/\varphi} = \varphi$ 
as $k$ tends to infinity.

\medskip   

Consequently, there are arbitrarily large integers $k$ such that $a_k = 2$. 

$\bullet$ If $a_{k+1} = 2$, $a_{k+2} = 2$, then $q_{k+2} = 5q_k + 2 q_{k-1}$, $q_{k+1} = 2q_k +q_{k-1}$, thus 
$$
\eta_{k+2} = \frac{2q_k + q_{k-1}}{5q_k + 2q_{k-1}} = \frac{2+\eta_k}{5+2\eta_k},
$$
$$
t_{k+2} = \frac{|W_k M_{k-1} M_k|}{q_{k+2}} =  \frac{|W_k| + q_k + q_{k-1}}{5q_k+2q_{k-1}}
= \frac{t_k+ 1+\eta_k}{5+2\eta_k}.
$$
By Lemma~\ref{lemt1} (ii), we get
\begin{equation*}
\frac{r(q_{k+2} + q_{k+1}-2,{\mathbf x})}{q_{k+2} + q_{k+1}-2}  < 1 + \frac{t_{k+2} +\eta_{k+2}}{1 + \eta_{k+2} } + \eps_{k} 
 = 1+\frac{t_k + 3 + 2\eta_k}{7+3\eta_k} + \eps_k < \varphi.
\end{equation*}

$\bullet$ If $a_k=1$, $a_{k+1} = 2$, $a_{k+2} = 1$, $a_{k+3} = 2$, then
we have
$$
q_{k+3} = 11q_{k-1} + 8q_{k-2},  \quad  q_{k+2} = 4q_{k-1} + 3q_{k-2}, 
$$
thus 
$$
\eta_{k+3} = \frac{4q_{k-1} + 3q_{k-2}}{11q_{k-1} + 8q_{k-2}} = \frac{4+3\eta_{k-1}}{11+8\eta_{k-1}},
$$
$$
t_{k+3} = \frac{|W_{k-1}|  + q_{k-2} + q_{k-1} + q_k + q_{k+1}}{q_{k+3}} 
= \frac{t_{k-1}+5+4\eta_{k-1}}{11+8\eta_{k-1}}.$$
By Lemma~\ref{lemt1} (i) we may assume that
$$
\frac{1+\eta_{k-1}}{t_{k-1}+1+\eta_{k-1}} > \varphi -1, \quad \hbox{that is,}
\qquad t_{k-1} < (\varphi-1)(1 + \eta_{k-1}).
$$
Using Lemma~\ref{lemt1} (ii), we get
\begin{equation*}
\frac{r(q_{k+3}+q_{k+2}-2,{\mathbf x})}{q_{k+3}+q_{k+2}-2} < 
  1 + \frac{t_{k+3} +\eta_{k+3}}{1 + \eta_{k+3} } + \eps_{k} 
 = 1+\frac{t_{k-1} + 9 + 7\eta_{k-1}}{15+11\eta_{k-1}} + \eps_{k}.
\end{equation*}
For $\eta_{k-1} \le \varphi -1$, we obtain
\begin{align*}
\frac{t_{k-1} + 9 + 7\eta_{k-1}}{15+11\eta_{k-1}} 
\le \frac{\varphi+8 + (\varphi + 6)\eta_{k-1}}{15+11\eta_{k-1}}&  \le 
\frac{\varphi+8 + (\varphi + 6)(\varphi -1)}{15+11(\varphi-1)} \\
& = \frac{7\varphi + 3}{11 \varphi + 4}= \frac{5\sqrt{5}+69}{122}= 0.6572\ldots
\end{align*}
For $\eta_{k-1}  > \varphi -1$, we get
$$
\frac{t_{k-1} + 9 + 7\eta_{k-1}}{15+11\eta_{k-1}} 
\le \frac{10 + 7\eta_{k-1}}{15+11\eta_{k-1}} < \frac{7 \varphi + 3}{11 \varphi + 4} 
= \frac{5\sqrt{5}+69}{122}= 0.6572\ldots < \sqrt{10} - \frac{5}{2}. 
$$

$\bullet$ If $a_{k+1} = 2$, $a_{k+2} = 1$, $a_{k+3} = 1$, $a_{k+4}=1$, then
by Lemma~\ref{lemt1} (ii) we may assume that
\begin{equation}\label{assum}
\frac{t_{k+1}+\eta_{k+1}}{1+\eta_{k+1}} = \frac{|W_{k+1}|+q_{k}}{q_{k+1}+q_{k}} = \frac{|W_{k}|+q_{k}+q_{k-1}}{3q_{k}+q_{k-1}}  = \frac{t_k+1+\eta_k}{3+\eta_k} > \varphi -1 .
\end{equation}
We have 
$q_{k+4} = 8q_k + 3q_{k-1}$, $q_{k+3} = 5q_k + 2q_{k-1}$, thus 
$$
\eta_{k+4} = \frac{5q_k + 2q_{k-1}}{8q_k + 3q_{k-1}} = \frac{5+2\eta_k}{8+3\eta_k},
$$
$$
t_{k+4} = \frac{|W_k| + q_{k-1} + q_k+q_{k+1}+q_{k+2}}{q_{k+4}} =  \frac{t_k+6+3\eta_k}{8+3\eta_k}.
$$
By Lemma~\ref{lemt1} (i), we get 
\begin{equation*}\begin{split}
\frac{r(|W_{k+4}|+q_{k+4}+q_{k+3}-2,{\mathbf x})}{|W_{k+4}|+q_{k+4}+q_{k+3}-2} 
&<   
1 + \frac{1+\eta_{k+4}}{t_{k+4} + 1 + \eta_{k+4}} +\eps_k \\
&= 1+\frac{13+ 5\eta_k}{t_k+19+8\eta_k} + \eps_k \\
&<1+\frac{13+ 5\eta_k}{(\varphi-1)(3+\eta_k)+18+7\eta_k} + \eps_k \\
&< 1+\frac{18}{21+4\varphi} + \eps_k\\
&= 1.6552\ldots + \eps_k < \sqrt{10} -\frac{3}{2},
\end{split}\end{equation*}
where we used the inequality \eqref{assum}.

Suppose that $\liminf_{n \to + \infty} \, \frac{r(n,{\mathbf x})}{n} \ge \sqrt{10} - \frac 32$.   
We have established that there exists  
an integer $K$ such that the slope of ${\mathbf x}$ is equal 
to $[0; a_1, a_2, \ldots, a_K, \overline{2,1,1}]$
and for all $k \ge K$   
$${\mathbf x} = W_{k+1} M_k M_{k+1} \tilde M_{k+1} \dots = W_k M_{k-1} M_k M_k \tilde M_k \dots.$$      
We establish now that, under these assumptions, we have  
$$
\liminf_{n \to + \infty} \, \frac{r(n,{\mathbf x})}{n} = \sqrt{10} - \frac 32. 
$$
Let $k$ be an integer with $k > K$. 
By Lemma~\ref{lemt1} (i),
\begin{equation*}
\frac{r(|W_{3k+K}|+q_{3k+K}+q_{3k-1+K}- 2,{\mathbf x})}{|W_{3k+K}| + q_{3k+K}+q_{3k-1+K}-2} < 
1+ \frac{1+\eta_{3k+K}}{t_{3k+K} + 1 + \eta_{3k+K} } + \eps_{3k+K}.
\end{equation*}
Since 
\begin{equation*}  
\eta_{k} = \frac{q_{k-1}}{q_{k}}  = [0;a_{k},a_{k-1}, \dots, a_1] 
\end{equation*}
and  
\begin{equation*}\begin{split} 
t_{k} &= \frac{q_{k-2} + q_{k-3} + \ldots + q_{K-1} + |W_K|}{q_k} \\
 &= \eta_{k} \eta_{k-1} + \eta_{k} \eta_{k-1}\eta_{k-2} + \ldots+ \eta_k \eta_{k-1} \cdots \eta_{K} + \frac{|W_K|}{q_k} ,   
\end{split}\end{equation*} 
 we check that
\begin{equation}\label{7.11}
\lim_{k \to + \infty} \eta_{3k+K} =  \frac{\sqrt{10}}{2}-1 \quad \text{ and } \quad 
\lim_{k \to + \infty} t_{3k+K} = \frac{8-\sqrt{10}}{6},
\end{equation}
giving that
$$
\liminf_{n \to + \infty} \, \frac{r(n,{\mathbf x})}{n} \le \sqrt{10}-\frac{3}{2}. 
$$

Let us now show that this inequality is indeed an equality.  

Since $M_k M_{k-1}$ is primitive,  
Lemma~\ref{factors}   
implies that 
all of the first $(q_k + q_{k-1})$ factors of length $(q_k + q_{k-1}-1)$    
of the word ${\mathbf x}=  W_k M_{k-1} M_k M_k \tilde M_k \dots$  
are distinct, 
thus we have 
\begin{equation*}\label{fbound2}    
r(q_k + q_{k-1}-1 , {\mathbf x}) \ge 2q_k + 2q_{k-1}-1.
\end{equation*}
The next $|W_k|$ factors of ${\mathbf x}$  
of length $(q_k + q_{k-1}-1)$ are identical with its first $|W_k|$ factors 
since,  for $1 \le i \le |W_k|$, we have  
$$ 
x_{i}^{i +q_k + q_{k-1}-2} =  x_{i + q_k + q_{k-1} }^{i +2q_k + 2q_{k-1} - 2}
= (W_k)_i^{|W_k|} (\tilde M_k)_1^{i +q_k + q_{k-1}- |W_k|-2}. 
$$  
By the fact that the last two letters of $M_k M_{k-1}$ and $M_{k-1} M_k$ are different,    
we get   
$$
x_{|W_k| +q_k + q_{k-1}-1} \ne  x_{|W_k| + 2q_k + 2q_{k-1}-1}. 
$$  
It follows that, for $1 \le i \le |W_k|$,  we have
$$ 
x_{i}^{i +|W_k| +q_k + q_{k-1}-2} \ne  x_{i + q_k + q_{k-1}}^{i +|W_k| + 2q_k + 2q_{k-1}-2}. 
$$   
Therefore, we get  
\begin{equation*}\label{fbound1}   
r( |W_k| +q_k + q_{k-1}-1 , {\mathbf x}) \ge 2|W_k| +2q_k + 2q_{k-1}-1.
\end{equation*}
%
It then follows from Lemma~\ref{prelim} (iii) that 
$$ r(n,{\mathbf x}) \ge 
\begin{cases}
 n + q_{k} + q_{k-1},  & q_{k} + q_{k-1}- 1 \le n \le |W_{k}|+q_{k}+q_{k-1}-2,\\
 n + |W_k| + q_k+q_{k-1},  & |W_k|+q_k+q_{k-1}-1 \le n \le q_{k+1} + q_k-2.
\end{cases}
$$

We also check that   
\begin{align*}
\lim_{k \to + \infty} \eta_{3k+1+K} &=  \frac{\sqrt{10}-2}{3}, &
\lim_{k \to + \infty} \eta_{3k+2+K} &=  \frac{\sqrt{10}-1}{3}. \\
\lim_{k \to + \infty} t_{3k+1+K} &=  \frac{8-\sqrt{10}}{9} , &
\lim_{k \to + \infty} t_{3k+2+K} &= \frac  23.
\end{align*}
Combined with \eqref{7.11} we get   
$$
\liminf_{k \to + \infty} \frac{|W_{k}|+2q_{k}+2q_{k-1}-2}{|W_{k}|+q_{k}+q_{k-1}-2}
= 1 + \liminf_{k \to + \infty} \frac{1 + \eta_k}{t_k + 1 + \eta_k}   
 = \sqrt{10} - \frac{3}{2}
$$
and
$$
\liminf_{k \to + \infty} \frac{|W_{k}|+q_{k+1}+2q_{k}+q_{k-1}-2}{q_{k+1}+q_{k}-2}
 = 1 + \liminf_{k \to + \infty} \frac{t_{k+1}+ \eta_{k+1}}{1+ \eta_{k+1}}   
= \frac{5}{3}. 
$$
Therefore, we conclude that  
\begin{equation*}
\rep({\mathbf x}) = \liminf_{n \to + \infty} \frac{r(n,{\mathbf x})}{n}  = \sqrt{10} - \frac{3}{2}.  
\end{equation*}
This completes the proof of Theorem~\ref{newthm}.

\smallskip

We remark that, in the course of the proof of Theorem~\ref{newthm}, we have established that 
if $\rep({\mathbf x}) <  \sqrt{10} - \frac{3}{2}$ for a Sturmian word $\mathbf x$, then 
$\rep({\mathbf x}) \le \frac{5\sqrt{5}+191}{122}= 1.6572\ldots$. 
Consequently, $\sqrt{10} - \frac{3}{2}$ is an isolated point of the set of real numbers 
$\rep(\bs)$, where $\bs$ runs over the Sturmian words.


\smallskip

\noindent {\it Examples of Sturmian words $\mathbf x$ 
such that $\rep({\mathbf x}) = \sqrt{10}-\frac 32$.}   

In the proof of Theorem~\ref{newthm} we have established that  
a Sturmian word $\bs'$ satisfies   
$$ \rep(\bs') = \sqrt{10} - \frac{3}{2} $$   
if and only if   
the continued fraction expansion of the slope of $\bs'$    
is eventually periodic and of the form $[0; a_1, a_2, \ldots, a_K, \overline{2,1,1}]$  
for some integer $K$  
and $\bs' = W_k M_{k-1} M_k \tilde M_k \dots$  
 for all sufficiently large $k$.  

Set $\theta = [0;a_1, a_2, \dots ] = [0; \overline{2,1,1}] = \frac{\sqrt{10}-2}{3} $.  
With $M_k$ defined as before,  for $k \ge 2$, the word $W_{k} =1 M_0 M_1 \ldots M_{k-2}$ 
 is a suffix of $M_k$. 
Define 
$$
\bs' = \lim_{k \to + \infty} W_k = \lim_{k \to + \infty} \big( 1 M_0 M_1 \ldots M_{k-2} \big) 
= 100101001001\dots .   
$$
By applying Theorem~1 and Proposition~1 of \cite{ArFeHu99} with $e_n = 1$ for $n\ge 1$,
we see that the intercept of $\bs'$ is equal to 
$$
(1-\theta)\left(1 +\sum_{n=0}^\infty (-1)^{n+1} \theta_1 \cdots \theta_{n+1} e_{n+1} \right)
 = 1 -\theta - \sum_{k=0}^\infty (q_k \theta - p_k) = \frac{1}{3},
$$    
where $\theta_1 = [0;a_1-1,a_2, \dots]$ and $\theta_k = [0;a_k, a_{k+1}, \dots]$.   

The example of Cassaigne \cite{Cassa97} for the minimal value of  
$\limsup_{n \to + \infty} \frac{R'(n)}{n}$ 
 is given by the fixed point of the substitution $\sigma$ defined by 
$$
\sigma (0) = 01001010, \qquad  \sigma(1) = 010.   
$$ 
Set
$$
\mathbf c := \lim_{k \to + \infty} \sigma^k (0) = 
0100101001001001010010010100100100101001001 \dots 
$$
The word $\mathbf c$ is a Sturmian word of slope $[0; 2, \overline{1, 2,1}]$.  
Let $(M^{\mathbf c}_k)_{k \ge 0}$   
be the corresponding sequence of words given by \eqref{mk_con}. 
Then it is easy to check by induction that 
$010 \sigma (M^{\mathbf c}_k) = M^{\mathbf c}_{k+3} 010$ for $k \ge 0$.   
Therefore,  we have  
\begin{align*}     
\sigma( 0 1 M^{\mathbf c}_0 M^{\mathbf c}_1 M^{\mathbf c}_2 \dots ) &=  01001010 \ 010 \  
\sigma (M^{\mathbf c}_0) \ \sigma(M^{\mathbf c}_1) \ \sigma(M^{\mathbf c}_2) \dots  \\   
&= 01 M^{\mathbf c}_0 M^{\mathbf c}_1 M^{\mathbf c}_2 \ 010 \sigma (M^{\mathbf c}_0) \ \sigma(M^{\mathbf c}_1) \sigma(M^{\mathbf c}_2) \dots  \\
&= 01 M^{\mathbf c}_0 M^{\mathbf c}_1 M^{\mathbf c}_2 M^{\mathbf c}_3 M^{\mathbf c}_4 \dots , 
\end{align*}
and it follows that 
$\mathbf c = 01 M^{\mathbf c}_0 M^{\mathbf c}_1 M^{\mathbf c}_2 M^{\mathbf c}_3 M^{\mathbf c}_4 \dots$,   
thus $\rep({\mathbf c}) = \sqrt{10} - \frac 32$.  

Let $\tau$ be the substitution given by $\tau(0) = 10$ and $\tau (1) = 0$.  
We check by induction that $0 \tau (M^{\mathbf c}_k) = M_{k+1} 0$ holds for all $k \ge 0$. 
We conclude that ${\mathbf c}$ and $\bs'$ are related by  
\begin{align*}   
\tau({\mathbf c})    
&= \tau(01 M^{\mathbf c}_0 M^{\mathbf c}_1 M^{\mathbf c}_2 M^{\mathbf c}_3 M^{\mathbf c}_4 \dots) \\  
&= 10\ 0 \tau(M^{\mathbf c}_0) \tau(M^{\mathbf c}_1) \tau ( M^{\mathbf c}_2) \tau( M^{\mathbf c}_3) \tau( M^{\mathbf c}_4) \dots \\ 
&= 10 M_1 M_2 M_3 M_4 M_5 \dots = \bs'. 
\end{align*}


\medskip

\noindent {\it Proof of Theorem~\ref{newthmquasi}.}

Let ${\mathbf y}$ be an infinite word defined over a finite alphabet ${\mathcal A}$ such that
the sequence $(p(n,{\mathbf y}) - n)_{n \ge 1}$ is bounded and ${\mathbf y}$ is not ultimately periodic. 
It follows from Theorem~\ref{MH} that the sequence $(p(n,{\mathbf y}) - n)_{n \ge 1}$ 
of positive integers is nondecreasing and bounded. Thus, it is
eventually constant. There exist positive integers $k$ and $n_0$ such that
$$
p(n,{\mathbf y}) = n + k, \quad \text{ for } n \ge n_0.
$$
It then follows from a result of Cassaigne \cite{Cassa98} that there are a 
finite word $W$, a Sturmian word $\bs$ defined over $\{0, 1\}$ and a 
morphism $\phi$ from $\{0, 1\}^*$ into ${\mathcal A}^*$ such that 
$\phi (01) \not= \phi (10)$ and 
$$
{\mathbf y} = W \phi (\bs).
$$
Write $\bs = s_1 s_2 \ldots $
Let $n$ be a large positive integer.
The word $V_n := s_{r(n,\bs)-n+1}^{r(n,\bs)}$ of length $n$ has two occurrences in $s_1^{r(n,\bs)}$. 
Consequently, the word $\phi(V_n)$ has two occurrences in the prefix of ${\mathbf y}$
of length $|W| + |\phi (s_1^{r(n,\bs)})|$, thus
$$
r(|\phi(V_n)|, {\mathbf y}) \le |W| + |\phi (s_1^{r(n,\bs)})|.
$$
A classical property of Sturmian words asserts that $0$ and $1$ have a frequency in $\bs$.
Consequently, by arguing as in \cite{Ad10}, there exists a real number $\delta$ such that
$$
|\phi (s_1 s_2 \ldots s_n)| = \delta n + o(n), \quad \hbox{for every $n \ge 1$}.  
$$
Let $\eps$ be a positive real number.
For $n$ large enough 
there exist real numbers $\eta_n$ and $\mu_n$ with $|\eta_n|, |\mu_n| \le \eps n$ and
$$
r(\delta n + \eta_n, {\mathbf y}) \le |W| + \delta r(n,\bs) + \mu_n.
$$
As $n$ can be taken arbitrarily large, this implies that
$$
\rep({\mathbf y}) = \liminf_{n \to + \infty} \, \frac{r(n,{\mathbf y})}{n}    
\le \frac{\delta}{\delta - \eps} \, \liminf_{n \to + \infty} \, \frac{r(n, \bs)}{n} + \frac{\eps}{\delta - \eps}.  
$$
Since $\eps$ can be chosen arbitrarily small, we deduce that
$$
\rep({\mathbf y})  \le \liminf_{n \to + \infty} \, \frac{r(n, \bs)}{n} = \rep(\bs).  
$$
In view of  Theorem~\ref{newthm}, this proves Theorem~\ref{newthmquasi}.

\section{Rational approximation}

In this section and in the next one, for a finite word $W$ 
and a real number $w \ge 1$, we write $W^w$
for the concatenation of $\lfloor w \rfloor$ copies of $W$ and the
prefix of length $\lceil ( w - \lfloor w \rfloor) |W| \rceil$ of $W$.  

\medskip

\noindent {\it Proof of Theorem \ref{minmurep}. }

\medskip

Since the irrationality exponent of an irrational real number is at least 
equal to $2$, we can assume that $\rep({\mathbf x}) < 2$. 
Let $n$ be a positive integer 
such that $r(n,{\mathbf x}) < 2 n$.  
By the theorem of Lyndon and Sch\"utzenberger (Theorem 1.5.2 in \cite{AlSh03}), 
this implies that there are finite words $W, U, V$ 
(we do not indicate the dependence on $n$) and a positive integer $t$ such that   
$|(UV)^t U| = n$ 
and $W (UV)^{t+1} U$ is the prefix of ${\bf x}$ of length $r(n,{\mathbf x})$.    
Observe that 
$$
|WUV| = |W (UV)^{t+1} U| - |(UV)^t U| = r(n,{\mathbf x}) - n.    
$$
Setting $\xi = \sum_{k \ge 1} \, \frac{x_k}{b^k}$, there exists an integer $s$ such that 
$\xi$ and the rational number $\frac{s}{b^{|W|} (b^{|UV|} - 1)}$ have the same $r(n,{\mathbf x})$   
first digits in their $b$-ary expansions, thus  
\begin{align*}
\Bigl| \xi - \frac{s}{b^{|W|} (b^{|UV|} - 1)} \Bigr| \le 
\frac{1}{b^{|W (UV)^{t+1} U|}} 
& = \frac{1}{b^{|WUV| + |(UV)^{t} U|}} \\
&= \frac{1}{b^{|WUV|} b^{n |WUV| / (r(n,{\mathbf x}) - n) }}.
\end{align*}
We derive that 
$$
\mu (\xi) \ge 1 + \limsup_{n \to + \infty} \, \frac{n}{r(n,{\mathbf x}) - n},
$$
thus, $\mu (\xi)$ is infinite if $\rep({\mathbf x}) = 1$ and
$$
\mu(\xi) \ge 1 + \frac{1}{\rep({\mathbf x})-1},
$$
otherwise. This proves the theorem.


\bigskip

\noindent {\it Proof of Theorem \ref{readirrexp}.}

\medskip

We assume that the reader is familiar with the theory of continued fractions
(see e.g. Section 1.2 of \cite{BuLiv}).  

Set   
$\xi := \sum_{k \ge 1} \, \frac{x_k}{b^k}$.  
Write   
$\xi = [0; d_1, d_2, \ldots]$ and let $(\frac{p_j}{q_j})_{j \ge 1}$ denote the
sequence of its convergents. 

Let ${\mathcal N} := (n_k)_{k \ge 1}$ be the increasing sequence of all the integers $n$ 
such that $r(n+1, {\bf x}) \ge r(n, {\bf x}) + 2$.
Let $k$ be a positive integer. 
By Lemma~\ref{jump} we have $r(n_k +1, {\bf x}) = 2 n_k + 3$. 

We deduce from the definition of the sequence ${\mathcal N}$ that 
\begin{equation}\label{eq8.0}
r(n_k + \ell, {\bf x}) = 2 n_k + 2 + \ell, \quad 1 \le \ell \le n_{k+1} - n_k.
\end{equation}

Set $\alpha_k = \frac{r(n_k, {\bf x})}{n_k}$. 
Observe that $\alpha_k \le 2 + \frac{1}{n_k}$ and
\begin{equation}\label{eq8.1}
\rep({\mathbf x}) = \liminf_{k \to + \infty} \, \alpha_k. 
\end{equation}

Let $k$ be an integer for which $\alpha_k < 2$  
(infinitely many such $k$ do exist since $\rep({\bf x}) < 2$).  
Let $W_k, U_k, V_k$ be the words associated with $n_k$    
as in the previous proof and $w_k, u_k, v_k$ 
their lengths, which satisfy 
$w_k + u_k + v_k = (\alpha_k - 1) n_k$.
There exists an integer $s_k$ such that   
the $\alpha_k n_k$ first digits of ${\bf x}$ and those of the $b$-ary expansion 
of the rational number $\frac{s_k}{b^{w_k} (b^{u_k + v_k} - 1)}$ coincide.
Consequently, we get
\begin{equation}\label{eq8.2}
\Bigl| \xi - \frac{s_k}{b^{w_k} (b^{u_k + v_k} - 1)} \Bigr| \le \frac{1}{b^{\alpha_k n_k}}.    
\end{equation}
A classical theorem of Legendre (see e.g. Theorem 1.8 of \cite{BuLiv}) 
asserts that, if the irrational real number $\zeta$ 
and the rational number $\frac{p}{q}$ with $q \ge 1$ 
satisfy $|\zeta - \frac{p}{q}|< \frac{1}{2 q^2}$,
then $\frac{p}{q}$ is a convergent of the continued fraction expansion of $\zeta$.

Since $\alpha_k < 2$, we get $\alpha_k \le 2 - \frac{1}{n_k}$. 
As
$$
2 \bigl( b^{w_k} (b^{u_k + v_k} - 1) \bigr)^2 
< 2 b^{2 (\alpha_k - 1) n_k} \le b^{\alpha_k n_k}
$$
holds if $\alpha_k n_k \le 2 n_k - 1$,  
Legendre's theorem and the assumption $\alpha_k < 2$ imply that 
the rational number $\frac{s_k}{b^{w_k} (b^{u_k + v_k} - 1)}$, which may not be 
written under its reduced form,
is a convergent, say $\frac{p_h}{q_h}$, of the continued fraction expansion of $\xi$.  


Let $\ell$ be the smallest positive integer such that $\alpha_{k + \ell} < 2$.

We first establish that $\ell \le 2$ if $n_k$ is sufficiently large.

Assume that $r(n_{k+1}, {\mathbf x}) = 2 n_{k+1} + \eps_{k+1}$ 
and $r(n_{k+2}, {\mathbf x}) = 2 n_{k+2} + \eps_{k+2}$, 
with $\eps_{k+1}, \eps_{k+2} \in \{0, 1\}$. 
Put $\eta_k := r(n_{k+2}, {\mathbf x}) - r(n_{k+1}, {\mathbf x})$. Since 
\begin{equation}\label{alpha}\begin{split}
\alpha_{k+2} n_{k+2} = r(n_{k+2}, {\bf x}) 
& = r(n_{k+1} + (n_{k+2} - n_{k+1}), {\bf x}) \\ 
& = 2n_{k+1} + 2 + n_{k+2} - n_{k+1}  = n_{k+2}  + n_{k+1} + 2, 
\end{split}\end{equation}
we get $n_{k+2} = n_{k+1} + 2 - \eps_{k+2}$, thus 
$$
\eta_k  = 2 (n_{k+1} + 2 - \eps_{k+2}) + \eps_{k+2} - 2 n_{k+1} -  \eps_{k+1} 
= 4 - \eps_{k+1} - \eps_{k+2}. 
$$
This shows that $\eta_k \in \{2, 3, 4\}$. 

By a well-known property of Sturmian sequences (see \cite{Loth02} on page 46), 
for any $n \ge 1$, there 
exists a unique factor $Z_n$ (called a right special factor) 
of ${\bf x}$ of length $n$ such that $Z_n 0$ 
and $Z_n 1$ are both factors of ${\bf x}$.

It follows from our assumption $r(n_{k+1} + 1, {\mathbf x}) > r(n_{k+1}, {\mathbf x}) + 1$ 
that $Z_{n_{k+1}} = {x}_{r(n_{k+1}, {\mathbf x})-n_{k+1}+1}^{r(n_{k+1}, {\mathbf x})}$. 
Likewise, we get 
$Z_{n_{k+2}} = {x}_{r(n_{k+2}, {\mathbf x})-n_{k+2}+1}^{r(n_{k+2}, {\mathbf x})}$, thus
$$
Z_{n_{k+1}} = { x}_{r(n_{k+1}, {\mathbf x})-n_{k+1}+1}^{r(n_{k+1}, {\mathbf x})}   
= { x}_{r(n_{k+2}, {\mathbf x})-n_{k+1}+1}^{r(n_{k+2}, {\mathbf x})}  
= { x}_{r(n_{k+1}, {\mathbf x}) + \eta_k -n_{k+1}+1}^{r(n_{k+1}, {\mathbf x}) + \eta_k}.  
$$
It then follows from the theorem of Lyndon and Sch\"utzenberger (Theorem 1.5.2 in \cite{AlSh03})   
that there exists an integer $t_k$, a
word $T_k$ of length $\eta_k$ and a prefix $T'_k$ of $T_k$ such that
$$
Z_{n_{k+1}} = (T_k)^{t_k} T'_k.
$$
We deduce that
$$
t_k \ge \frac{n_{k+1} - 3}{4}.   
$$ 
Since $|T_k| \le 4$ and  
a Sturmian word cannot contain unbounded powers of a fixed word (see \cite[Corollary 10.6.6]{AlSh03}),  
there exists an integer $t$ such that no factor of ${\bf x}$ is a $t$-th power of a word of length 
less than or equal to 4.  

Consequently, if $k$ is large enough, then we cannot have simultaneously 
$r(n_{k+1}, {\mathbf x}) \ge 2 n_{k+1}$ and $r(n_{k+2}, {\mathbf x}) \ge 2 n_{k+2}$. 
This implies that $\ell = 1$ or $\ell = 2$.

Since $\alpha_{k + \ell} < 2$, it follows from Legendre's theorem that the rational number 
$\frac{s_{k+\ell}}{b^{w_{k+\ell}} (b^{u_{k+\ell} + v_{k+\ell}} - 1)}$, 
which may not be 
written under its reduced form, is a convergent, say $\frac{p_j}{q_j}$, 
of the continued fraction expansion of $\xi$.  
The $(\alpha_k n_k + 1)$-th digit of the $b$-ary expansion of $\frac{p_j}{q_j}$ is equal to 
the $(\alpha_k n_k + 1)$-th digit of ${\mathbf x}$ 
and differs from the $(\alpha_k n_k + 1)$-th digit 
of the $b$-ary expansion of $\frac{p_h}{q_h}$. Consequently, 
the rational numbers $\frac{p_h}{q_h}$ and $\frac{p_j}{q_j}$ are distinct. 

Here, the indices $h$ and $j$ depend on $k$.   
We have 
\begin{equation}\label{eq8.4}
q_h \le b^{w_k} (b^{u_k + v_k} - 1) \le b^{(\alpha_k-1)n_k}
\end{equation}
and
$$
q_j \le b^{w_{k+\ell}} (b^{u_{k+\ell} + v_{k+\ell}} - 1) \le b^{(\alpha_{k+\ell}-1)n_{k+\ell}}.    
$$ 
Note that it follows from \eqref{alpha} that
$$
(\alpha_{k+2} - 1) n_{k+2} = n_{k+1} + 2.    
$$
and, likewise,
$$
(\alpha_{k+1} - 1) n_{k+1} = n_k + 2,    
$$ 
Note that $n_{k+1} \le n_k + 2$ if $\alpha_{k+1} \ge 2$. 

The properties of continued fractions give that
\begin{equation}\label{eq8.5}
\frac{1}{2 q_h q_{h+1}} \le \Bigl| \xi - \frac{p_h}{q_h} \Bigr| \le \frac{1}{q_h q_{h+1}}
\end{equation}
and
$$
\frac{1}{2 q_j q_{j+1}} \le \Bigl| \xi - \frac{p_j}{q_j} \Bigr| \le \frac{1}{q_j q_{j+1}}. 
$$
This implies that  
$$
q_{j+1} \ge \frac{b^{\alpha_{k+\ell} n_{k+\ell}}}{2 q_j} \ge \frac{b^{n_{k+\ell}}}{2}.
$$
Since $\alpha_k < 2$, we get
$$
q_h \le b^{(\alpha_{k} - 1) n_{k}} \le b^{n_k - 1} < \frac{b^{n_{k+\ell}}}{2} \le q_{j+1}. 
$$
Combined with $p_h / q_h \not= p_j / q_j$,
this gives 
$$
q_h < q_{h+1} \le q_j < q_{j+1}.
$$

It follows from
$$
q_{h} \ge \frac{b^{\alpha_k n_k}}{2 q_{h+1}}
$$
and
\begin{equation}\label{eq8.6}
q_{h+1} \le q_j \le b^{(\alpha_{k+\ell} - 1) n_{k+\ell}} \le b^{n_k + 4},  
\end{equation}
that 
\begin{equation}\label{eq8.7}
q_h \ge 
\frac{b^{\alpha_k n_k}}{2 b^{n_k + 4}}. 
\end{equation} 
Since $q_h \le b^{w_k} (b^{u_k + v_k} - 1) \le b^{(\alpha_{k} - 1) n_{k}}$,  
this shows that the rational number 
$\frac{s_k}{b^{w_k} (b^{u_k + v_k} - 1)}$
is not far from being reduced, 
in the sense that the greatest common divisor of its numerator and denominator 
is at most equal to $2 b^4$. 
Furthermore, we deduce from (\ref{eq8.2}), (\ref{eq8.4}), (\ref{eq8.5}), 
(\ref{eq8.6}), and (\ref{eq8.7}) that 
\begin{equation}\label{eq8.8}
\frac{1}{(2 b^4 q_h)^{\alpha_k / (\alpha_k - 1)}} \le \Bigl| \xi - \frac{p_h}{q_h} \Bigr| 
\le \frac{1}{q_h^{\alpha_k / (\alpha_k - 1)}}.
\end{equation}

Moreover, it follows from
$$
q_{h+1} \ge \frac{b^{\alpha_k n_k}}{2 q_h} \ge \frac{b^{n_k}}{2}
$$
that
$$
1 \le \frac{q_j}{q_{h+1}} \le 2 b^4.
$$
Consequently, all the partial quotients $d_{h+2}, \ldots , d_j$  are less than $2 b^4$ 
and we get 
$$
\Bigl| \xi - \frac{p_\ell}{q_\ell} \Bigr| > \frac{1}{(q_\ell+q_{\ell+1}) q_\ell} 
 > \frac{1}{(d_{\ell + 1} + 2) q_{\ell}^2} \ge
\frac{1}{2 (b^{4} + 1) q_{\ell}^2},  
$$
for $\ell = h+1, \ldots , j-1$.  

Now, we are armed to conclude the proof. 
We consider the increasing sequence ${\mathcal K}$ of integers $k$ such that $\alpha_k < 2$. 
Let $k$ be an element of ${\mathcal K}$ and assume that $k$
is sufficiently large. We have established that there exist integers 
$h(k)$ and $j(k)$ such that 
all the partial quotients $d_{h(k) +2}, \ldots , d_{j(k)}$  are less than $2 b^4$. 
Furthermore, (\ref{eq8.8}) provides us with a precise estimate of
$d_{h(k) + 1}$.
The definitions of $h$ and $j$ show that if $k'$ is the next element after $k$ in 
the sequence ${\mathcal K}$, then $h(k') = j(k)$. Consequently, we have a precise 
estimate of all but finitely many partial quotients of $\xi$ and we deduce
from (\ref{eq8.1}) and (\ref{eq8.8}) that 
$$
\mu (\xi) = \limsup_{k \to + \infty} \, \frac{\alpha_k}{\alpha_k - 1} = 
\frac{\rep({\mathbf x})}{\rep({\mathbf x}) - 1}. 
$$
This completes the proof of the theorem.


\section{On the recurrence function of an infinite word}\label{sec9}

Cassaigne \cite{Cassa97} studied the recurrence function $n \mapsto R'(n)$
of an infinite word $\mathbf x = x_1 x_2 \dots $, which is defined 
as the length of the shorted prefix of $\mathbf x$ containing an occurrence 
of every factor of $\mathbf x$ of length $n$. 
Then it is not difficult to check that $R'(n) \ge p(n, \mathbf x) + n-1$ 
and the equality holds if and only if $r(n, \mathbf x) = p(n, \mathbf x) + n$.
Moreover, for a Sturmian word $\mathbf x$,  
we have the following relation between $r(n, \mathbf x)$ and $R'(n)$. 

\begin{proposition}\label{prop:relation}
For any Sturmian word $\mathbf x$, we have 
$$ 
\limsup_{n \to +\infty} \frac{R'(n)}{n} = \frac{\rep(\mathbf x)}{\rep(\mathbf x) -1}.    
$$
\end{proposition}

Therefore, it follows from Theorem~\ref{newthm} that  
$$
\limsup_{n \to +\infty} \frac{R'(n)}{n} \ge \frac 53 + \frac{4 \sqrt{10}}{15} = 2.5099\dots ,  
$$
and this value is optimal.

\begin{proof}
Let $\mathbf x = x_1 x_2 x_3 \dots $ be a Sturmian word.
Let $n$ be a positive integer such that $R'(n) \ge 2n+1$.
Since $p(n,\mathbf x) = n+1$, there exist integers $i, j$ such that
$$
0 \le i < j \le R'(n) -n  
\quad \hbox{and} \quad x_{i+1}^{i+n} = x_{j+1}^{j+n}.
$$
It follows from the definition of $R'(n)$ that 
$x_{R'(n)-n+1}^{R'(n)}$ is not a factor of $x_1^{R'(n)-1}$.
Thus,  there exists $m \ge 0$ such that
$$
x_{i+1}^{i+n+m} = x_{j+1}^{j+n+m}, \quad x_{i+n+m+1} \ne x_{j+n+m+1}, 
\quad \hbox{and} \quad j+n+m+1 \le R'(n).
$$
Therefore, $x_{i+m+1}^{i+n+m+1}$ and  $x_{j+m+1}^{j+n+m+1}$ are the 
two factors of ${\mathbf x}$ of length $n+1$ 
extending the right special factor $x_{i+m+1}^{i+m+n}$,  
and $x_{1}^{R'(n)+1}$ contains all the factors of $\mathbf x$ of length $n+1$. 
This shows that $R'(n+1) =R'(n) +1$ whenever $R'(n) \ge 2n+1$. 

Let $(n_k)_{k \ge 1}$ be the increasing sequence of all the integers $n$ 
such that $r(n+1, \mathbf x) \ge r(n, \mathbf x) +2$.
It then follows from \eqref{eq8.0} that 
$$ 
\rep(\mathbf x)
= \liminf_{k \to +\infty}  \frac{r(n_k,\mathbf x)}{n_k}
=\liminf_{k \to +\infty}  \frac{n_k + n_{k-1} +2}{n_k} = 1 + \liminf_{k \to +\infty} \frac{n_{k-1}}{n_k}.
$$
For every positive integer $n$, we have $R'(n) = 2n$ if, and only if, 
$r(n, \mathbf x) = 2n +1$. This shows that $R'(n_k + 1) = 2 (n_k+1)$ holds for every
positive integer $k$. Furthermore, we have established above that $R'(n+1) = R'(n) +1$ 
if $n$ is not an element of the sequence $(n_k +1)_{k \ge 1}$.  
Consequently, we have 
\begin{align*}
\limsup_{n \to +\infty} \frac{ R'(n)}{n} &= \limsup_{k \to +\infty}  \frac{ R'(n_k +2)}{n_k +2}   
=\limsup_{k \to +\infty}  \frac{n_{k+1} + n_k +3}{n_k +2} \\    
&= 1 + \limsup_{k \to +\infty} \frac{n_{k+1}}{n_k} = 1 + \frac{1}{\rep(\mathbf x) -1}.
\end{align*}
This proves the proposition. 
\end{proof}


\section{Links with other combinatorial exponents}

There are various combinatorial exponents associated 
with infinite words. 
One of them, the initial critical exponent, was introduced in 
2006 by Berth\'e, Holton, and Zamboni \cite{BeHoZa06}.

\begin{definition}\label{defice}
The initial critical exponent of an infinite word $\mathbf x$, 
denoted by $\ice({\mathbf x})$, is the supremum of 
the real numbers $\rho$ for which there exist arbitrary long prefixes $V$ of $\mathbf x$  
such that $V^\rho$ is a prefix of $\mathbf x$.
\end{definition} 

The definition of the Diophantine exponent of an infinite word 
appeared in \cite{AdBu07a}, but this notion was implicitly used in earlier works 
of the same authors. 

\begin{definition}
The Diophantine exponent of an infinite word $\mathbf x$, 
denoted by $\dio({\mathbf x})$, is the supremum of 
the real numbers $\rho$ for which there exist arbitrary long prefixes of $\mathbf x$ that can be
factorized as $U V^w$, with $U$ and $V$ finite words and $w$ a real number such that
$$
\frac{|UV^w|}{|UV|} \ge \rho.
$$
\end{definition}

It follows from Definitions 9.1 and 9.2 that every infinite word $\mathbf x$ satisfies
$$
1 \le \ice({\mathbf x}) \le \dio({\mathbf x}) \le + \infty.   \leqno (9.1)
$$
Furthermore, there are words $\mathbf x$ such that $\ice({\mathbf x}) < \dio({\mathbf x})$.

The following lemma shows that the Diophantine exponent and the exponent of repetition 
are closely related.

\begin{lemma}\label{diorep}
Let $\mathbf x$ be an infinite word written over a finite alphabet. We have
$\rep({\mathbf x})=1$ (resp. $= + \infty$) if and only if $\dio({\mathbf x}) = + \infty$ (resp. $=1$).
Furthermore, if $1 < \dio({\mathbf x}) < + \infty$, then we have
$$
\rep ({\mathbf x}) = \frac{\dio({\mathbf x})}{\dio({\mathbf x}) - 1} 
\le \frac{\ice({\mathbf x})}{\ice({\mathbf x}) - 1}.
$$
\end{lemma}

\begin{proof}
In view of (9.1), it only remains for us to prove the first equality.    
To see that $\rep ({\mathbf x}) \le \frac{\dio({\mathbf x})}{\dio({\mathbf x}) - 1}$ 
it suffices to note that
if $UV^w$ is a prefix of $\mathbf x$, where $w > 1$ is chosen such that 
$|U V^w| = |U| + w |V|$, then
$$
\frac{r(|V|^{w-1},{\mathbf x})}{|V|^{w-1}} \le 
\frac{|UV^w|}{|V|^{w-1}} \le
\frac{|UV^w| / |UV|}{(|UV^w| / |UV|) - 1}.
$$

Conversely, if $r(n,{\mathbf x}) = Cn$ for some rational number $C$ and some integer $n$, then
the prefix of $\mathbf x$ of length $Cn$ can be written under the form $U V^w$, where $w >1$
and $|UV| = (C-1)n$. This implies that $\dio({\mathbf x}) \ge \frac{C}{C-1}$. 
Letting $C$ tend to $\rep({\mathbf x})$, 
we get $\dio({\mathbf x}) (\rep({\mathbf x}) - 1) \ge \rep({\mathbf x})$, that is,
$\rep({\mathbf x}) \ge \frac{\dio({\mathbf x})}{\dio({\mathbf x}) - 1}$. 
\end{proof}

One motivation for considering the function $n \mapsto r(n,{\mathbf x})$ comes from
Diophantine approximation. Indeed, the following transcendence criteria have
been recently established in \cite{ABL,AdBu07d,BuEv08,Bu13}, although they were
not highlighted in these papers, in which 
the subword complexity function $n \mapsto p(n,{\mathbf x})$ 
occurs in place of $n \mapsto r(n,{\mathbf x})$. 

\begin{theorem}\label{transsum}
Let $\mathcal A$ be a finite set of integers. 
Let ${\mathbf x}= x_1 x_2 \ldots$ be an infinite word over $\mathcal A$, which is
not eventually periodic. If
$$
\liminf_{n \to + \infty} \, \frac{r(n,{\mathbf x})}{n} < + \infty,
$$
or if there exists a real number $\eta$ with $\eta < 1/11$ and 
$$
\limsup_{n \to + \infty} \, \frac{r(n,{\mathbf x})}{n (\log n)^{\eta}} < + \infty,
$$
then, for every integer $b \ge 2$, the real number
$
\sum_{k \ge 1} \, \frac{x_k}{b^k}
$
is transcendental.
\end{theorem}

Recall that a real number is algebraic of degree two
if and only if its continued fraction expansion is eventually periodic. 

\begin{theorem}\label{transcf}
Let $\mathcal A$ be a finite set of positive integers. 
Let ${\mathbf x}= x_1 x_2 \ldots$ be an infinite word over $\mathcal A$. 
If $\mathbf x$ is not eventually periodic and
$$
\liminf_{n \to + \infty} \, \frac{r(n,{\mathbf x})}{n} < + \infty, 
$$
then the real number
$
[0; x_1, x_2 , \ldots]
$
is transcendental.
\end{theorem}

The interested reader is referred to the survey \cite{Bu14}, where the combinatorial 
assumption made on the infinite word $\mathbf x$ is precisely the following (the
same assumption is made in \cite{AdBu07d,BuEv08,Bu13}): 
we suppose that $\mathbf x$ is not eventually periodic and that there exist 
three sequences of finite words $(U_n)_{n \ge 1}$,   
$(V_n)_{n \ge 1}$ and $(W_n)_{n \ge 1}$ such that:

\begin{itemize}
\item[\rm (i)]  
For every $n \ge 1$, the word $W_n U_n V_n U_n$ is a prefix
of the word $\mathbf x$;

\item[\rm (ii)] 
The sequence
$({\vert V_n\vert} / {\vert U_n\vert})_{n \ge 1}$ is bounded from above;

\item[\rm (iii)] 
The sequence
$({\vert W_n\vert} / {\vert U_n\vert})_{n \ge 1}$ is bounded 
from above;

\item[\rm (iv)] 
The sequence $(\vert U_n\vert)_{n \ge 1}$ is 
increasing.

\end{itemize}

One sees that this assumption exactly means that $\dio({\mathbf x})$ exceeds $1$
and is, by Lemma \ref{diorep}, equivalent to the one made in the above transcendence criteria.  
Using Lemma~\ref{ubound}, we deduce immediately that $r(n,{\mathbf x})$ can be replaced
by $p(n,{\mathbf x})$ in 
Theorems \ref{transsum} and \ref{transcf}. Consequently, Lemma 8.1 of \cite{Bu14}  
(which was also used in \cite{AdBu07d,Bu13})
is not needed to deduce Theorems 3.1 and 3.2 of \cite{Bu14} from the
combinatorial transcendence criteria stated in Section 4 of that paper. 
This shows that considering the function $n \mapsto r(n,{\mathbf x})$ is
indeed the right point of view.

We end this section with a theorem established in \cite{AdBu11}.
It is stated in that paper with the subword complexity function 
$n \mapsto p(n,{\mathbf x})$, but, in that 
paper as well, the 
proofs actually work if this function is replaced by $n \mapsto r(n,{\mathbf x})$. 
For the definition of Mahler's classification, the reader 
is directed to Chapter 3 of \cite{BuLiv}.

\begin{theorem}
Let $\xi$ be a real number such that its expansion $\mathbf x$ in some integer base $b \ge 2$
satisfies 
$$
\limsup_{n \to + \infty} \, \frac{r(n,{\mathbf x})}{n} < + \infty.
$$
If $\rep({\mathbf x}) = 1$, then $\xi$ is a Liouville number. Otherwise,  
$\xi$ is either an $S$-number or a $T$-number in Mahler's classification. 
\end{theorem}

\section*{Acknowledgement}
Dong Han Kim was supported by the National Research Foundation of Korea (NRF-2015R1A2A2A01007090). 
The authors wish to warmly thank the anonymous referee for very careful reading and 
many valuable comments.

\end{document}